\documentclass[a4paper,11pt]{article}

\usepackage[T1]{fontenc}
\usepackage{lmodern}

\usepackage{dsfont}

\usepackage[latin1]{inputenc}
\usepackage{amsmath}
\usepackage{amsthm}
\usepackage{amssymb}
\usepackage{mathrsfs}
\usepackage{graphicx}
\usepackage[all]{xy}
\usepackage{hyperref}

\usepackage{makeidx}

\usepackage{stmaryrd}
\usepackage{caption}

\usepackage{abstract} 

\newtheorem{thm}{Theorem}[section]
\newtheorem{cor}[thm]{Corollary}
\newtheorem{claim}[thm]{Claim}
\newtheorem{fact}[thm]{Fact}

\newtheorem{lemma}[thm]{Lemma}
\newtheorem{prop}[thm]{Proposition}

\theoremstyle{definition}
\newtheorem{definition}[thm]{Definition}
\newtheorem{ex}[thm]{Example}

\newtheorem{question}[thm]{Question}

\def\rquotient#1#2{%
	\makeatletter
	\raise.3ex\hbox{$#1$}/\lower.3ex\hbox{$#2$}%
	\makeatother
}	

\makeatletter
\newcommand{\subjclass}[2][2010]{%
	\let\@oldtitle\@title%
	\gdef\@title{\@oldtitle\footnotetext{#1 \emph{Mathematics subject classification.} #2}}%
}
\newcommand{\keywords}[1]{%
	\let\@@oldtitle\@title%
	\gdef\@title{\@@oldtitle\footnotetext{\emph{Key words and phrases.} #1.}}%
}
\makeatother

\newcommand{\Address}{{
		\bigskip
		\small
		
\noindent\textsc{University of Montpellier\\ 
Institut Math\'ematiques Alexander Grothendieck\\
Place Eug\`ene Bataillon\\
34090 Montpellier (France)}\par\nopagebreak
\noindent\textit{E-mail address}: \texttt{anthony.genevois@umontpellier.fr}
		
}}

\makeindex

\title{A median degree from crossing graphs of median graphs}
\date{\today}
\author{Anthony Genevois}


\begin{document}

\maketitle

\begin{abstract}
The \emph{crossing graph} $\mathrm{Cross}(M)$ of a median graph $M$ is defined as the graph whose vertices are the $\Theta$-classes of $M$ and whose edges connect two $\Theta$-classes whenever they cross. It is known that every graph $X$ can be realised as the crossing graph of some median graph. In this article, we initiate the study of the space $\mathrm{Cross}^{-1}(X)$ of all the median graphs with crossing graph $X$. First, we prove that two finite median graphs have isomorphic crossing graphs if and only if one can be obtained from the other by a sequence of elementary transformations we call \emph{slidings}.  Then, motivated by the fact that $\mathrm{Cross}^{-1}(X)$ always contains a single median graph of maximal degree, namely the simplex-graph of $X$, we introduce the \emph{median degree} of $X$ as the smallest possible degree of a median graph in $\mathrm{Cross}^{-1}(X)$. We compute the median degree for some families of graphs and characterise the graphs with maximal median degree. 
\end{abstract}

\tableofcontents

\section{Introduction}

\noindent
Given a connected graph $X$, two edges $\{a,b\}$ and $\{x,y\}$ are in \emph{$\Theta$-relation} whenever
$$d(a,x)+d(b,y) \neq d(a,y)+d(b,x).$$
The relation $\Theta$, also called the \emph{Djokovi\'{c}-Winkler relation}, originates from the study of \emph{partial cubes}, i.e.\ graphs that can be isometrically embedded into hypercubes \cite{MR314669,MR727925}. More generally, it turns out to be tightly connected to isometric embeddings into Cartesian products of graphs \cite{MR776391}. 

\medskip \noindent
\begin{minipage}{0.48\linewidth}
\includegraphics[width=0.95\linewidth]{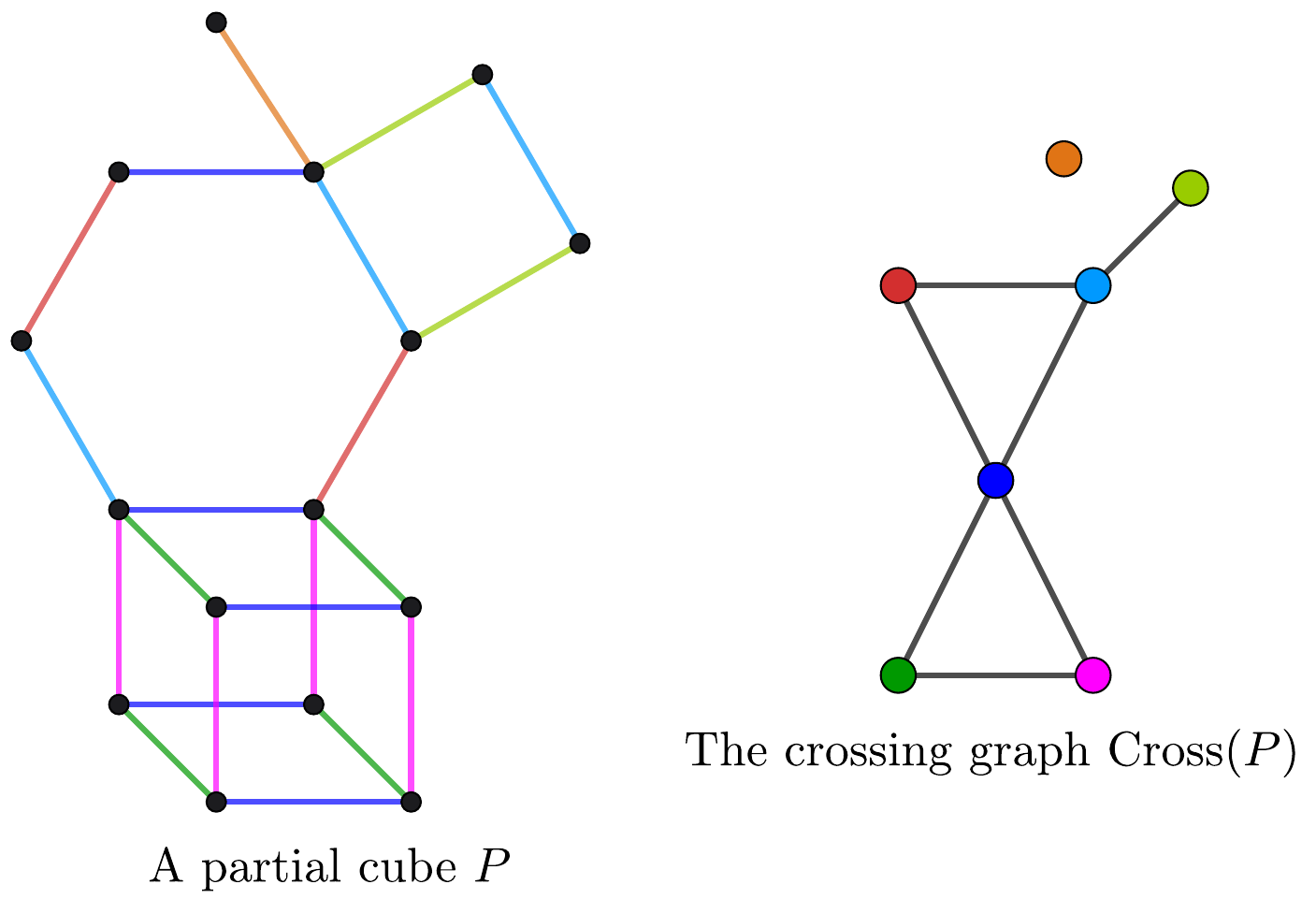}
\end{minipage}
\begin{minipage}{0.5\linewidth}
Motivated by the central role played by the relation $\Theta$ in the study of partial cubes, \cite{MR1920184} associates to every partial cube $P$ its \emph{crossing graph} $\mathrm{Cross}(P)$, namely the graph whose vertices are the $\Theta$-classes of edges of $P$ and whose edges connect two distinct $\Theta$-classes whenever they \emph{cross}, i.e.\ when they both intersect a common isometrically embedded cycle.  
\end{minipage}

\medskip \noindent
In this article, we focus on crossing graphs of median graphs. Recall that a connected graph $M$ is \emph{median} if, for all vertices $x_1,x_2,x_3 \in V(M)$, there exists a unique vertex $m \in V(M)$, referred to as the \emph{median point}, satisfying
$$d(x_i,x_j)= d(x_i,m)+d(m,x_j) \text{ for all } i \neq j.$$
Median graphs, such as products of trees and hypercubes, are notable examples of partial cubes. They play a remarkable role in the study of crossing graphs of partial cubes due to the following observation: every graph can be described as the crossing graph of some median graph \cite{MR1920184} (see also \cite{Roller,MR3217625}). This article is dedicated to some aspects of the following (vague) question:

\begin{question}
Given a graph $X$, what can be said about 
$$\mathrm{Cross}^{-1}(X) := \{ \text{median graphs $M$ satisfying } \mathrm{Cross}(M) \simeq X \}?$$
\end{question}

 \noindent
As just said, we know that the set $\mathrm{Cross}^{-1}(X)$ is never empty. But how big it is? How are the graphs it contains related? Are there natural representatives? When $X$ is an $n$-clique $K_n$, then it is not difficult to show that $\mathrm{Cross}^{-1}(X)= \{ n\text{-cube } Q_n\}$. However, when $X$ is given by $n$ isolated vertices, then $\mathrm{Cross}^{-1}(X)= \{ \text{trees with $n$ edges}\}$. So, in general, there may exist many median graphs with a given crossing graph. 

\medskip \noindent
As the first main result of this article, we state and prove a characterisation of finite median graphs having the same crossing graph. For this, we introduce a graph-theoretic transformation we call a \emph{sliding}. 

\begin{thm}\label{thm:IntroMain}
Two finite median graphs have the same crossing graph if and only if one can be obtained from the other by a sequence of slidings.
\end{thm}
\begin{figure}
\begin{center}
\includegraphics[width=0.7\linewidth]{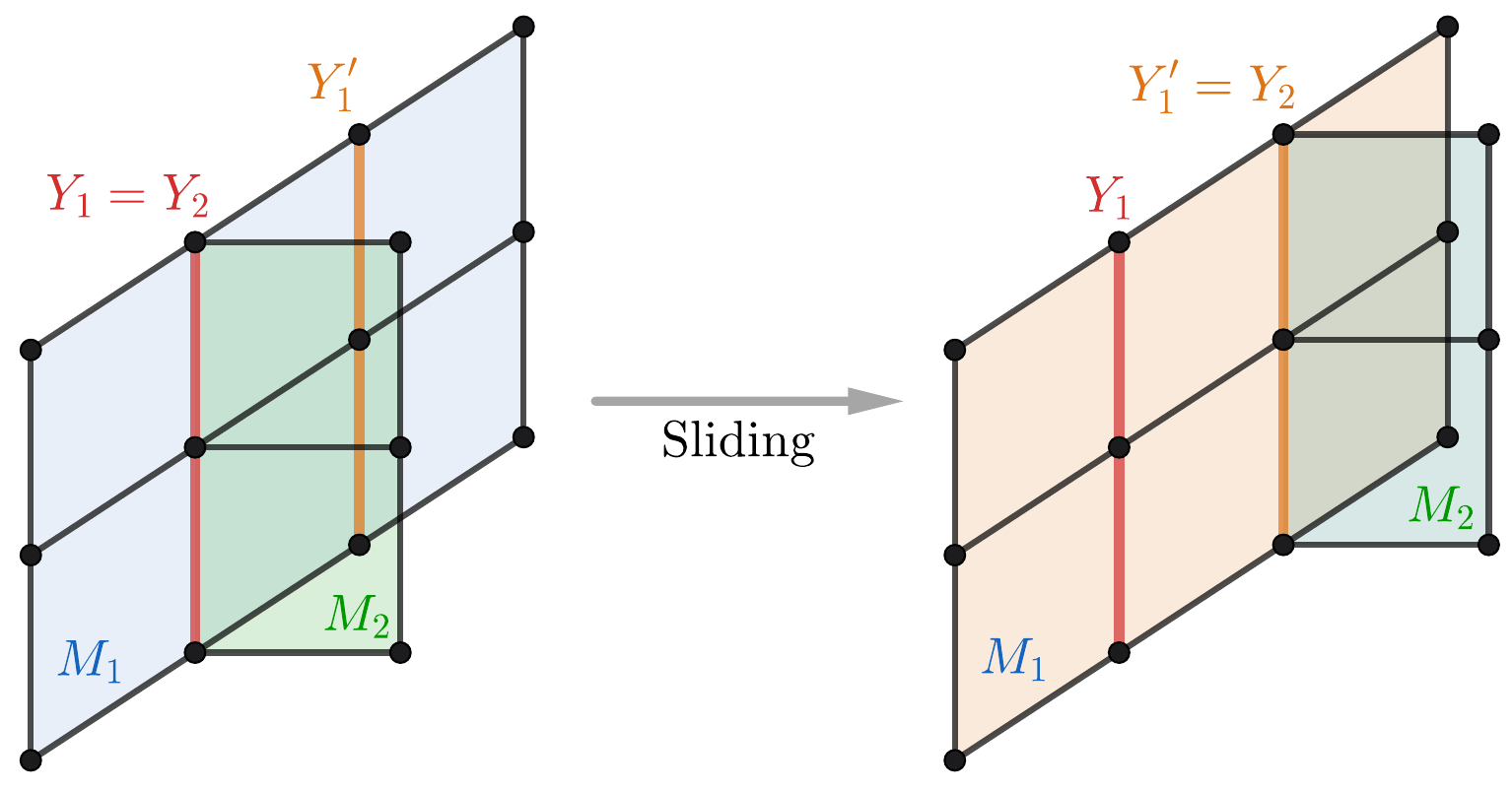}
\caption{Sliding in a median graph.}
\label{SlidingIntro}
\end{center}
\end{figure}

\noindent
In a nutshell, a sliding will transform a median graph $M$ decomposing into an amalgam $A \ast_{C_1 = C_2} B$ into a new amalgam $A \ast_{C_1'=C_2}B$ by ``sliding'' $B$ along $A$ from $C_1$ to $C_1'$, where, somehow, the subgraphs $C_1$ and $C_1'$ are ``parallel''. See Figure~\ref{SlidingIntro} for a concrete example and Section~\ref{section:Sliding} for a precise definition. 

\medskip \noindent
Then, we turn to the following question: given a graph $X$, is there a simplest representative in $\mathrm{Cross}^{-1}(X)$? In view of the example given below, illustrating two median graphs $M_1$ and $M_2$ with the same crossing graph, it seems that there may not be a reasonable choice, no obvious reason to prefer some representatives than the others.
\begin{center}
\includegraphics[width=\linewidth]{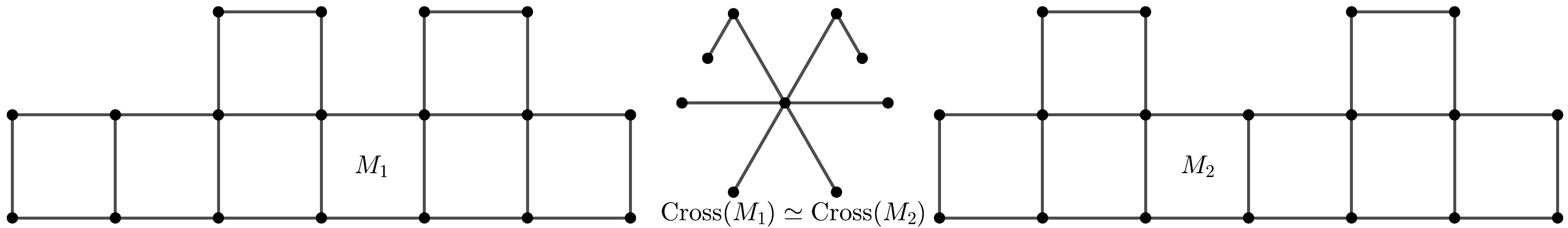}
\end{center}

\medskip \noindent
On the other hand, in some sense, there is a most complicated representative. As shown in \cite{MR1920184}, given an arbitrary graph $X$, the \emph{simplex-graph} $\mathrm{Simp}(X)$ turns out to be a median graph whose crossing is isomorphic to $X$. Here, $\mathrm{Simp}(X)$ refers to the graph whose vertices are the (possibly empty) cliques of $X$ and whose edges connect two cliques whenever one can be obtained from the other by adding or removing a single vertex. It is not difficult to show that the simplex-graph $\mathrm{Simp}(X)$ is the unique median graph in $\mathrm{Cross}^{-1}(X)$ that has maximal degree, namely $|V(X)|$ (see Proposition~\ref{prop:MaxDegree}). This observation motivates us to consider the degree as a measure of the complexity of a representative in $\mathrm{Cross}^{-1}(X)$.

\begin{definition}
The \emph{median degree} $\mathrm{mdeg}(X)$ of a graph $X$ is the smallest degree of a median graph $M$ whose crossing graph $\mathrm{Cross}(M)$ is isomorphic to $X$. 
\end{definition}

 \noindent
Considering crossing graphs of median graphs with bounded degree is also motivated by \cite{MR4057355,MR4862336}, in which we proved that the crossing graph of a median graph with no cut-vertex and of degree $\leq N$ is always a quasi-tree with a hyperbolicity constant that depends only on $N$. 

\medskip \noindent
In Section~\ref{section:Example}, we record some computations of median degrees for explicit examples of graphs such as complete graphs and cycles (Corollary~\ref{cor:CliqueGirth}), grids (Proposition~\ref{prop:Grid}), and bipartite complete graphs (Corollary~\ref{cor:BiComplete}). 

\medskip \noindent
As the second main result of this article, we characterise the graphs with maximal median degree:

\begin{thm}\label{thm:IntroMain}
Let $X$ be a finite graph. The following assertions are equivalent:
\begin{itemize}
	\item the equality $\mathrm{mdeg}(X)= |V(X)|$ holds;
	\item there exists a unique median graph whose crossing graph is $X$, namely $\mathrm{Simp}(X)$;
	\item  $X$ cannot be separated by the star of one of its vertices. 
\end{itemize}
\end{thm}

\noindent
In Section~\ref{section:Questions}, we conclude the article by leaving a few open questions related to median graphs and their crossing graphs.

\section{Preliminaries on median graphs}

\noindent
A \emph{graph} $X$ is the data of a set of \emph{vertices} $V(X)$ and a set of \emph{edges} $E(X) \subseteq \{ \{x,y \} \in x,y \in V(X) \text{ distinct} \}$. A \emph{map} $\beta : X \to Y$ between two graphs $X$ and $Y$ is the data of two maps $V(X) \to V(Y)$ and $E(X) \to E(Y)$, which we also denote by $\beta$ for simplicity. The map $\beta$ is a \emph{graph-morphism} if $\{\beta(x),\beta(y)\} \in E(Y)$ for every edge $\{x,y \} \in E(X)$, and a \emph{graph-isomorphism} if in addition the induced maps $V(X) \to V(Y)$ and $E(X) \to E(Y)$ are bijective. 

\medskip \noindent
Given a graph $X$, two vertices $x,y \in V(X)$ are \emph{adjacent} if $\{x,y\} \in E(X)$. One also says that $y$ is a \emph{neighbour} of $x$. The \emph{star} of our vertex $x$ is the subgraph induced by $x$ and its neighbours. The \emph{degree} of a vertex is its (possibly infinite) number of neighbours. The \emph{degree} of a graph is the (possibly infinite) supremum of the degrees of its vertices. 

\medskip \noindent
A graph is \emph{connected} if, for all vertices $x,y \in V(X)$, there exists a sequence $v_0:=x,v_1, \ldots, v_{n-1},v_n:=y$ such that $\{v_i,v_{i+1}\} \in E(X)$ for every $0 \leq i \leq n-1$. The smallest integer $n$ we can choose is the \emph{distance} between $x$ and $y$. This allows us to endow every connected graph with a metric. A path of minimal length is a \emph{geodesic}. A subgraph containing all the geodesics between its vertices is \emph{convex}. 

\medskip \noindent
A connected graph $M$ is \emph{median} if, for all vertices $x_1,x_2,x_3 \in V(M)$, there exists a unique vertex $m \in V(M)$, referred to as the \emph{median}, satisfying 
$$d(x_i,x_j)=d(x_i,m)+d(m,x_j) \text{ for all distinct } 1 \leq i,j \leq 3.$$
Examples of median graphs include products of trees, including (hyper)cubes. Recall that, given a set $S$, the \emph{(hyper)cube} $Q(S)$ is the graph whose vertices are the finite subsets of $S$ and whose edges connect two subsets whenever one is obtained from the other by adding or removing a single point. Given a median graph $M$, the largest cardinality of a set $S$ for which $M$ contains an induced subgraph isomorphic to $Q(S)$ is the \emph{cubical dimension} of $M$. Notice that median graphs of cubical dimension $1$ coincide with trees.

\medskip \noindent
Cubes play a central role in the structure of median graphs. We record for future two properties satisfied by any median graph $M$:
\begin{description}
	\item[($3$-cube condition)] For every vertex $v \in V(M)$ and all neighbours $a,b,c \in V(M)$, if the edges $\{v,a\}$, $\{v,b\}$, $\{v,c\}$ pairwise span a square, then they globally span a $3$-cube.
	\item[(cube condition)] For every vertex $v \in V(M)$ and all neighbours $a_1, \ldots, a_k \in V(M)$, if the edges $\{v,a_1\}, \ldots, \{v,a_k\}$ pairwise span a square, then they globally span a $k$-cube.
\end{description}
Here, a \emph{square} refers to an induced $4$-cycle. See for instance \cite{MR1748966} for more details.

\medskip \noindent
In median graphs, convex subgraphs are \emph{gated}. Recall that a subgraph $Y \leq X$ is \emph{gated} if, for every vertex $x \in V(X)$, there exists a vertex $p \in V(Y)$ such that $d(x,y)=d(x,p)+d(p,y)$ for every $y \in V(y)$. The vertex $p$ is referred to as the \emph{gate} or the \emph{gate-projection} of $x$. 

\medskip \noindent
Given a connected graph $X$, two edges $\{a,b\}$ and $\{x,y\}$ are in \emph{$\Theta$-relation} whenever
$$d(a,x)+d(b,y) \neq d(a,y)+d(b,x).$$
For median graphs, the relation $\Theta$ coincides with the reflexive-transitive closure of the relation that identifies two edges whenevery they are opposite in a square \cite{MR1210081}. Two distinct $\Theta$-classes \emph{cross} whenever they both intersect a square, or equivalently, when they contain two representatives that span a square. Two distinct $\Theta$-classes are \emph{in contact} if they contain two representatives that share an endpoint. (Here, we borrow the terminology from \cite{MR3217625}.)  

\medskip \noindent
The next statement summarises the central properties satisfied by $\Theta$-classes in median graphs. All the assertions are fairly standard. We refer the reader for instance to \cite{MR1788124} for more details. 

\begin{thm}\label{thm:BigHyp}
Let $M$ be a median graph. The following assertions are satisfied:
\begin{itemize}
	\item Every $\Theta$-class separates $M$ into two pieces, i.e.\ the graph $M \backslash \backslash J$ obtained from $M$ by removing the edges from $J$ has exactly two connected components, called \emph{halfspaces}. Halfspaces are convex.
	\item The \emph{carrier} of a $\Theta$-class $J$ is the subgraph of $M$ induced by the edges of $J$. A \emph{fibre} of $J$ is a connected component of $\partial J:= N(J) \backslash \backslash J$. Carriers and fibres are convex. 
product structure  and convexity of carriers.
	\item  For every $\Theta$-class $J$, every edge $e \in J$, and every fibre $F$ of $J$, the map $$x \mapsto (\text{gate-projection of $x$ to $e$}, \text{gate-projection of $x$ to $F$})$$ induces a graph-isomorphism $N(J) \to e \times F$. 
	\item A path in $M$ has minimal length if and only if it crosses every $\Theta$-class at most once. Consequently, the distance between two vertices coincides with the number of $\Theta$-classes separating them.
\end{itemize}
\end{thm}

\noindent
In the rest of the section, we collect a few preliminary results about median graphs that will be useful in the sequel. 

\begin{lemma}\label{lem:InterEdge}
Let $M$ be a median graph and $a,b \in E(M)$ two edges with a common endpoints. The following assertions hold.
\begin{itemize}
	\item The edges $a$ and $b$ belong to the same $\Theta$-class if and only if they coincide. 
	\item The edges $a$ and $b$ span a square in $M$ if and only if they belong to crossing $\Theta$-classes. 
\end{itemize}
\end{lemma}

\begin{proof}
Let $\{u,v\},\{v,w\} \in E(M)$ be two edges with a common endpoint $v$. If $u \neq w$, then, since $M$ is bipartite, $u,v,w$ must be a geodesic. Then, Theorem~\ref{thm:BigHyp} implies that the $\Theta$-classes of $\{u,v\}$ and $\{v,w\}$ are distinct. This proves the first assertion of our lemma.

\medskip \noindent
Assume that the $\Theta$-classes of $\{u,v\}$ and $\{v,w\}$ cross. Let $A$ (resp.\ $B$) denote the halfspace delimited by the $\Theta$-class of $\{u,v\}$ (resp.\ $\{v,w\}$) that does not contain $v$. Since our $\Theta$-classes cross, we know that $A \cap B \neq \emptyset$. Fix a vertex $z$ in this intersection and let $m$ denote the median point of $\{u,w,m\}$. By convexity of halfspaces (as given by Theorem~\ref{thm:BigHyp}), $m$ also belongs to $A \cap B$, hence $d(m,u),d(m,w) \geq 1$. Since $d(u,w)=2$, necessarily $m$ is a common neighbour of $u$ and $w$. In other words, $\{v,u\}$ and $\{v,w\}$ span a square in $M$, proving the second assertion of our lemma. 
\end{proof}

\begin{cor}\label{cor:ContactDegree}
Let $M$ be a median graph. If $M$ contains $d$ pairwise in contact $\Theta$-classes, then $\mathrm{deg}(M) \geq d$. 
\end{cor}

\begin{proof}
Let $J_1, \ldots, J_d$ be $d$ pairwise in contact $\Theta$-classes. We know that the carriers $N(J_1), \ldots, N(J_d)$ pairwise intersect. Due to the convexity of carriers and to the Helly property satisfied by convex subgraphs in median graphs, the global intersection $N(J_1) \cap \cdots \cap N(J_d)$ must be non-empty. In other words, there exists a vertex $x \in V(M)$ that belongs to an edge of $J_k$ for every $1 \leq k \leq d$. We deduce from Lemma~\ref{lem:InterEdge} that these $d$ edges are pairwise distinct, hence $\mathrm{deg}(M) \geq \mathrm{deg}(x) \geq d$, as desired. 
\end{proof}

\noindent
The following separation property can be found in \cite{MR721770}:

\begin{prop}\label{prop:MedianSepProp}
Let $X$ be a median graph and $A,B \leq X$ two convex subgraphs. If $A \cap B = \emptyset$, then there exists a $\Theta$-class separating $A$ and $B$.
\end{prop}

\begin{cor}\label{cor:SeparatingCarriers}
Let $X$ be a median graph and $J_1,J_2$ two $\Theta$-classes. If $J_1$ and $J_2$ are not in contact, then there exists a $\Theta$-class separating the carriers $J_1$ and $J_2$.
\end{cor}

\begin{proof}
Saying that $J_1$ and $J_2$ are not in contact amounts to saying that their carriers do not intersect. Since carriers are convex according to Theorem~\ref{thm:BigHyp}, our corollary follows from Proposition~\ref{prop:MedianSepProp}. 
\end{proof}

\noindent
The next statement can also be extracted from \cite{MR721770}. See for instance \cite[Proposition~2.7]{MR4398242} for a precise statement. 

\begin{prop}\label{prop:HypSepConvex}
Let $X$ be a median graph and $A,B \leq X$ two convex subgraphs. If $\alpha$ is a path of minimal length connecting $A$ to $B$, then the $\Theta$-classes crossed by $\alpha$ coincide with the $\Theta$-classes separating $A$ and $B$. 
\end{prop}

\noindent
Recall from the introduction that:

\begin{definition}
The \emph{crossing graph} $\mathrm{Cross}(M)$ of a median graph $M$ is the graph whose vertices are the $\Theta$-classes in $M$ and whose edges connect two $\Theta$-classes whenever they cross.
\end{definition}

\noindent
The following observation from \cite{MR1920184} will be needed later:

\begin{prop}\label{prop:ConnectedCrossing}
A median graph has a connected crossing graph if and only if it has no cut-vertex.
\end{prop}

\noindent
Let us mentioned another characterisation of having (or not having) a cut-vertex, based on the following definition:

\begin{definition}
Let $M$ be a median graph and $x \in V(M)$ a vertex. The \emph{cubical link} of $x$ is the simplicial complex $\mathrm{link}_\square(x)$ whose vertices are the neighbours of $x$ and whose simplices are the collections $\{x_1, \ldots, x_n\}$ such that the edges $\{x,x_1\}, \ldots, \{x,x_n\}$ form the corner of an $n$-cube in $M$.
\end{definition}

\noindent
It follows from \cite[Claim~3.3]{MR4862336} that:

\begin{prop}\label{prop:CutVertexLink}
Let $M$ be a median graph. A vertex $x \in V(M)$ is a cut-vertex if and only if its cubical link is disconnected.
\end{prop}

\section{Median graphs with the same crossing graph}

\subsection{Simplex graphs}

\noindent
As already mentioned in the introduction, it is known that every graph $X$ can be described as the crossing graph of some median graph $M(X)$. See for instance \cite[Theorem~3.1]{MR1920184}, \cite[Proposition~3.9]{Roller}, \cite[Proposition~2.19]{MR3217625}. In \cite{MR1920184}, a simple model for the median graph $M(X)$ is given, namely the \emph{simplex graph} of $X$.

\begin{definition}
Let $X$ be a graph. The \emph{simplex graph} $\mathrm{Simp}(X)$ is the graph whose vertices are the (possibly empty) cliques of $X$ and whose edges connect two cliques whenever one can be obtained from the other by adding or removing a single vertex. 
\end{definition}

\noindent
The terminology is taken from \cite{MR988098}. The following assertion can be found in \cite[p.\ 200]{MR874238}, where it is observed that simplex graphs are naturally covering graphs of median semilattices. An explicit statement can be found in \cite[Proposition~2.2]{MR988098}, where it is shown that a simplex graph can be easily described as a connected and median-closed subgraph in a hypercube. 

\begin{prop}
Simplex graphs of connected graphs are median.
\end{prop}

\noindent
Simplex graphs will play a central role in the proofs of our main results. Below, we record some information regarding $\Theta$-classes in simplex graphs. 

\begin{lemma}\label{lem:SimpHyp}
Let $X$ be a connected graph. The following assertions hold:
\begin{itemize}
	\item For every $x \in V(X)$, let $J_x$ denote the $\Theta$-class of $\{\emptyset, \{x\}\}$. The map $x \mapsto J_x$ induces a bijection from the vertices of $X$ to the $\Theta$-classes of $\mathrm{Simp}(X)$.
	\item For every $x \in V(X)$, $J_x = \{ \{ S, S \backslash\{x\}\} \mid S \text{ clique containing } x\}$ and $V(N(J_x))= \{ S \text{ clique} \mid S \cup \{x\} \text{ clique}\}$. 
	\item For every $x \in V(X)$, the vertex-sets of the fibres of $J_x$ are $\{ S \text{ clique} \mid x \in V(S) \}$ and $\{ S \backslash \{x\} \mid x \in V(S)\}$. 
\end{itemize}
\end{lemma}

\begin{proof}
Our lemma will be essentially  a consequence of the following observation:

\begin{claim}\label{claim:SimpTheta}
Let $x,y \in V(X)$ be two vertices. For every clique $S$ containing $x$, the edge $\{S, S \backslash \{x\}\}$ is in the same $\Theta$-class as $\{\emptyset, \{y\}\}$ if and only if $y=x$. 
\end{claim}

\noindent
Fix an enumeration $x,x_1, \ldots, x_k$ of the vertices of $S$. It follows from the following geometric configuration that the edge $\{S,S \backslash \{x\}\}$ lies in the same $\Theta$-class as $\{\emptyset, \{x\}\}$.
\begin{center}
\includegraphics[width=0.7\linewidth]{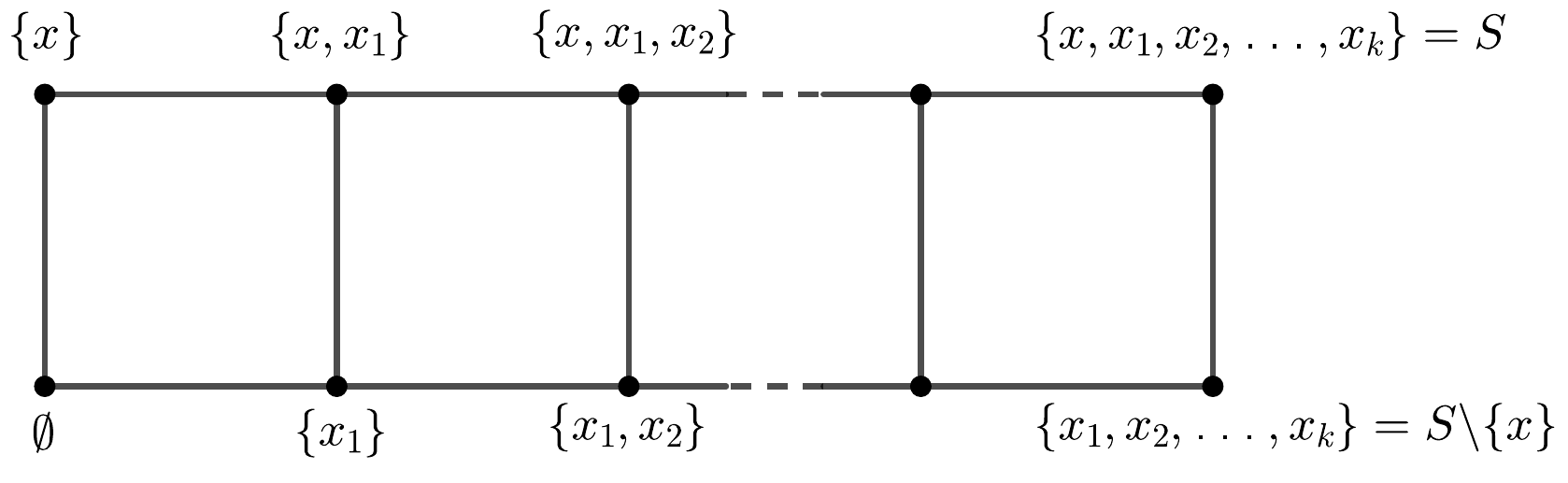}
\end{center}
Notice that, as a consequence of Lemma~\ref{lem:InterEdge}, two edges $\{\emptyset, \{x\}\}$ and $\{\emptyset, \{y\}\}$ belong to the same $\Theta$-class if and only if they agree, i.e.\ $x=y$. Thus, Claim~\ref{claim:SimpTheta} is proved. 

\medskip \noindent
It follows from Claim~\ref{claim:SimpTheta} that every $\Theta$-class is of the form $J_x$ for some $x \in V(X)$ and that two $\Theta$-classes $J_x$ and $J_y$ coincide if and only if $x=y$. Therefore, the map $x \mapsto J_x$ induces a bijection from the vertices of $X$ to the $\Theta$-classes of $\mathrm{Simp}(X)$, as desired. Claim~\ref{claim:SimpTheta} also immediately implies that, for every $x \in V(X)$, $J_x = \{ \{ S, S \backslash\{x\}\} \mid S \text{ clique containing } x\}$. Moreover, since the vertex-set of the carrier $N(J_x)$ coincides with the set of all the endpoints of the edges in $J_x$, the equality $V(N(J_x))= \{ S \text{ clique} \mid S \cup \{x\} \text{ clique}\}$ is clear. The description of the fibres of $J_x$ then follows from the following observation:

\begin{claim}
Let $x \in V(X)$ be a vertex. Two cliques $S_1$ and $S_2$ are separated by $J_x$ if and only if $x$ belongs to the symmetric difference $V(S_1) \triangle V(S_2)$. 
\end{claim}

\noindent
A path $\alpha$ in $\mathrm{Simp}(X)$ from $S_1$ to $S_2$ amounts to adding and removing vertices, say $v(1), \ldots, v(k)$ successively, in order to get $S_2$ from $S_1$. If $v(i)=v(j)$ for some distinct $1 \leq i,j \leq k$, then it follows from Claim~\ref{claim:SimpTheta} that $\alpha$ crosses some $\Theta$-class twice, and consequently cannot be a geodesic. Thus, if we choose $\alpha$ as a geodesic, then $v(1), \ldots, v(k)$ are pairwise distinct. Since $S_2$ is obtained from $S_1$ by adding or removing $v(1),\ldots, v(k)$ successively, it follows that $V(S_1) \triangle V(S_2)= \{v(1), \ldots, v(k)\}$. 

\medskip \noindent
Now, our $\Theta$-class $J_x$ separates $S_1$ and $S_2$ if and only if it is crossed by $\alpha$, which amounts to saying that $x \in \{v(1), \ldots, v(k)\}$, or equivalently that $x \in V(S_1) \triangle V(S_2)$ according to the equality just proved. 
\end{proof}

\noindent
We record for future use the following characterisation of simplex graphs, which is also of independent interest.

\begin{lemma}\label{lem:PairwiseContact}
A finite median graph is a simplex graph if and only if its $\Theta$-classes are pairwise in contact. 
\end{lemma}

\begin{proof}
Let $X$ be a connected graph. According to Lemma~\ref{lem:SimpHyp}, the edges $\{\emptyset, \{x\}\}$ for $x \in V(X)$ provide representatives for the $\Theta$-classes in $\mathrm{Simp}(X)$. Since they all share $\emptyset$ as an endpoint, they are clearly pairwise in contact. Thus, $\mathrm{Simp}(X)$ is a median graph all of whose $\Theta$-classes are in contact.

\medskip \noindent
Conversely, let $M$ be a finite median graph whose $\Theta$-classes are pairwise in contact. Because carriers of $\Theta$-classes are gated, it follows from the Helly property satisfied by gated subgraphs that $M$ contains some vertex, say $x_0$, that belongs to the carriers of all the $\Theta$-classes. We claim that the map
$$x \mapsto \mathcal{H}(x_0|x):= \{ \Theta\text{-classes separating $x_0$ and $x$} \}$$
induces a graph-isomorphism $\Xi : M \to \mathrm{Simp}(\mathrm{Cross}(M))$. First, let us verify that our map is well-defined, i.e.\ $\mathcal{H}(x_0|x)$ is a collection of pairwise crossing $\Theta$-classes for every $x \in V(M)$. If not, then there exists a vertex $x \in V(M)$ that is separated from $x_0$ by two non-crossing $\Theta$-classes, say $A$ and $B$. Up to renaming our $\Theta$-classes, say that $A$ separates $x_0$ from $B$. But, then, $x_0$ cannot belong the carrier of $B$, contradicting our definition of $x_0$. So we do have a well-defined map $\Xi$ from the vertices of $M$ tot he vertices of $\mathrm{Simp}(\mathrm{Cross}(M))$. Then, notice that
$$\mathcal{H}(x_0|x) \triangle \mathcal{H}(x_0|y) = \mathcal{H}(x|y) \text{ for all } x,y \in V(M)$$
where $\triangle$ denotes the symmetric difference. This implies that $\Xi$ is injective, since, given two vertices $x,y \in V(M)$, $\Xi(x)=\Xi(y)$ amounts to saying that $\mathcal{H}(x_0|x) = \mathcal{H}(x_0|y)$, hence  $\mathcal{H}(x|y)= \mathcal{H}(x_0|x) \triangle \mathcal{H}(x_0|y) = \emptyset$ and finally $x=y$. This also shows that $\Xi$ preserves the adjacency and the non-adjacency of vertices. Indeed, two vertices $x,y \in V(M)$ are adjacent in $M$ if and only if $|\mathcal{H}(x|y)|=1$, which amounts to saying that $|\mathcal{H}(x_0|x) \triangle \mathcal{H}(x_0|y)|=1$ or equivalently that $\mathcal{H}(x_0|x)$ can be obtained from $\mathcal{H}(x_0|y)$ by adding or removing a single vertex, which exactly means that $\mathcal{H}(x_0|x)$ and $\mathcal{H}(x_0|y)$ are adjacent in $\mathrm{Simp}(\mathrm{Cross}(M))$.

\medskip \noindent
It remains to verify that $\Xi$ is surjective. So let $J_1, \ldots, J_n$ be pairwise crossing $\Theta$-classes. Because $x_0$ belongs to the carriers of all these $\Theta$-classes, there exist neighbours $z_1, \ldots, z_n$ of $x_0$ such that, for every $1 \leq i \leq n$, $\{x_0,z_i\} \in J_i$. Given two distinct indices $1 \leq i,j \leq n$, the fact that $J_i$ and $J_j$ cross implies that the edges $\{x_0,z_i\}$ and $\{x_0,z_j\}$ span a square (Lemma~\ref{lem:InterEdge}). The cube condition implies that $x_0,z_1, \ldots, z_n$ span an $n$-cube $Q$. Let $x$ denote the vertex of $Q$ opposite to $x_0$. By construction, $\Xi(x)= \mathcal{H}(x_0|x) = \{J_1, \ldots, J_n\}$. This proves that $\Xi$ is indeed surjective. 
\end{proof}

\noindent
Finally, we conclude this section with the following observation, which shows that the median degree is always bounded above by the number of vertices and that the simplex graph is the only median graph that has this maximal degree. 

\begin{prop}\label{prop:MaxDegree}
Let $X$ be a finite graph. Every median graph $M$ whose crossing graph is isomorphic to $X$ has degree $\leq |V(X)|$. Moreover, $\mathrm{deg}(M)=|V(X)|$ if and only if $M$ is isomorphic to the simplex graph $\mathrm{Simp}(X)$. 
\end{prop}

\begin{proof}
Let $M$ be a median graph having $X$ has its crossing graph. Given a vertex $x \in V(M)$ of maximal degree and its neighbours $x_1, \ldots, x_n$, we know from Lemma~\ref{lem:InterEdge} that the edges $\{x,x_1\},\ldots, \{x,x_n\}$ belong to pairwise distinct $\Theta$-classes, so the degree of $x$ is at most the number of $\Theta$-classes in $M$, which coincides with the number of vertices of $X$. Hence $\mathrm{deg}(M) = \mathrm{deg}(x) \leq |V(X)|$. If there is equality, then $\{x,x_1\},\ldots, \{x,x_n\}$ must be representatives of all the $\Theta$-classes of $M$, which implies that $M$ has its $\Theta$-classes pairwise in contact. We deduce from Lemma~\ref{lem:PairwiseContact} that $M$ is isomorphic to $\mathrm{Simp}(\mathrm{Cross}(M)) = \mathrm{Simp}(X)$. Conversely, if $M$ isomorphic to $\mathrm{Simp}(X)$, then we deduce from the next observation that $\mathrm{deg}(M)=|V(X)|$.

\begin{claim}
The degree of $\mathrm{Simp}(X)$ is $|V(X)|$.
\end{claim}

\noindent
Let $S$ be a clique in $X$. The degree of $S$ in $\mathrm{Simp}(X)$ is the number of vertices in $S$ (which can be removed from $S$ to get smaller cliques) plus the number of vertices that are not in $S$ but that are adjacent to all the vertices in $S$ (which can be added to $S$ to get larger cliques). Clearly, $\mathrm{deg}(S) \leq |V(X)|$. On the other hand, $\empty$ is adjacent to $\{x\}$ for every $x \in V(X)$, so it has degree exactly $|V(X)|$. We conclude that $\mathrm{deg}(\mathrm{Simp}(X)) = |V(X)|$, as desired.
\end{proof}

\subsection{Slidings}\label{section:Sliding}

\noindent
In this section, we describe \emph{slidings} of median graphs, the central graph-theoretic operation underlying our characterisation of median graphs with isomorphic crossing graphs. In order to give a precise definition, we need to introduce some notation and terminology. 

\medskip \noindent
In the sequel, we will use the following notation. Given two graphs $A,B$, two subgraphs $C_1\leq A, C_2 \leq B$, and a graph-isomorphism $\varphi : C_1 \to C_2$, we denote by $A \ast_{C_1=_\varphi C_2} B$ the amalgam
$$\left( A \sqcup B \right) / \left( a = \varphi(a) \ \forall a \in V(A), \ e = \varphi(e) \ \forall e \in E(A) \right).$$
We recall that amalgamating two median graphs along isomorphic convex subgraphs still yields a median graph. 

\medskip \noindent
Roughly speaking, a sliding will transform an amalgam $A \ast_{C_1 = C_2} B$ into a new amalgam $A \ast_{C_1'=C_2}B$ by ``sliding'' $B$ along $A$ from $C_1$ to $C_1'$. Here, somehow the subgraphs $C_1$ and $C_1'$ have to be parallel copies of the same subgraph. In order to make precise this intuition, we need the following definition:

\begin{definition}
Let $M$ be a median graph. Two convex subgraphs $Y,Z \leq M$ are \emph{parallel} if they are crossed by the same $\Theta$-classes. 
\end{definition}

\noindent
The terminology is justfied by the following characterisation of parallelism:

\begin{prop}
Let $M$ be a median graph. Two convex subgraphs 
\begin{itemize}
	\item[(i)] $Y$ and $Z$ are parallel;
	\item[(ii)] the gate-projection to $Z$ induces a graph-isomorphism $Y \to Z$;
	\item[(iii)] there exists an isometric embedding $Y \times [0,n] \hookrightarrow M$ that sends $Y \times \{0\}$ to $Y$ and $Y\times\{n\}$ to $Z$.
\end{itemize}
\end{prop}

\begin{proof}
For the implication $(i) \Rightarrow (iii)$, see for instance \cite[Lemma~1.7]{MR4093881}. The implication $(iii) \Rightarrow (ii)$ is clear. Now, assume that $(ii)$ holds. If $J$ is a $\Theta$-class crossing $Z$, then there exists an edge $\{a,b \} \in E(Z)$ that belongs to $J$. Since the gate-projection to $Z$ induces a graph-isomorphism $Y \to Z$, there must be an edge $\{a',b'\}$ in $Y$ that projects to $\{a,b\}$. As a consequence of \cite[Proposition~2.6]{MR4071367}, the $\Theta$-class of $\{a',b'\}$ coincides with $J$. A fortiori, $J$ crosses $Y$. Conversely, if $J$ crosses $Y$ and if $\{p,q\}$ is an edge of $Y$ that belongs to $J$, then the edge of $Z$ on which $\{p,q\}$ projects must belong to $J$ again according to \cite[Proposition~2.6]{MR4071367}. A fortiori, $J$ crosses $Z$ as well. Thus, we have proved that $Y$ and $Z$ are crossed by the same $\Theta$-classes, i.e.\ $Y$ and $Z$ are parallel, as desired. 
\end{proof}

\noindent
It is worth mentioning that in (non-degenerate) median graphs, there is always a lot of parallelism:

\begin{lemma}\label{lem:FibreParallel}
Let $M$ be a median graph and $J$ a $\Theta$-class. The two fibres of $J$ are parallel.
\end{lemma}

\begin{proof}
A $\Theta$-class of $M$ crosses a fibre of $J$ if and only if it crosses $J$. So the two fibres of $J$ are crossed by the same $\Theta$-classes, namely the $\Theta$-classes crossing $J$. 
\end{proof}

\noindent
We are finally ready to define slidings of median graphs. We refer the reader to Figure~\ref{Sliding} for an explicit example. 
\begin{figure}
\begin{center}
\includegraphics[width=0.7\linewidth]{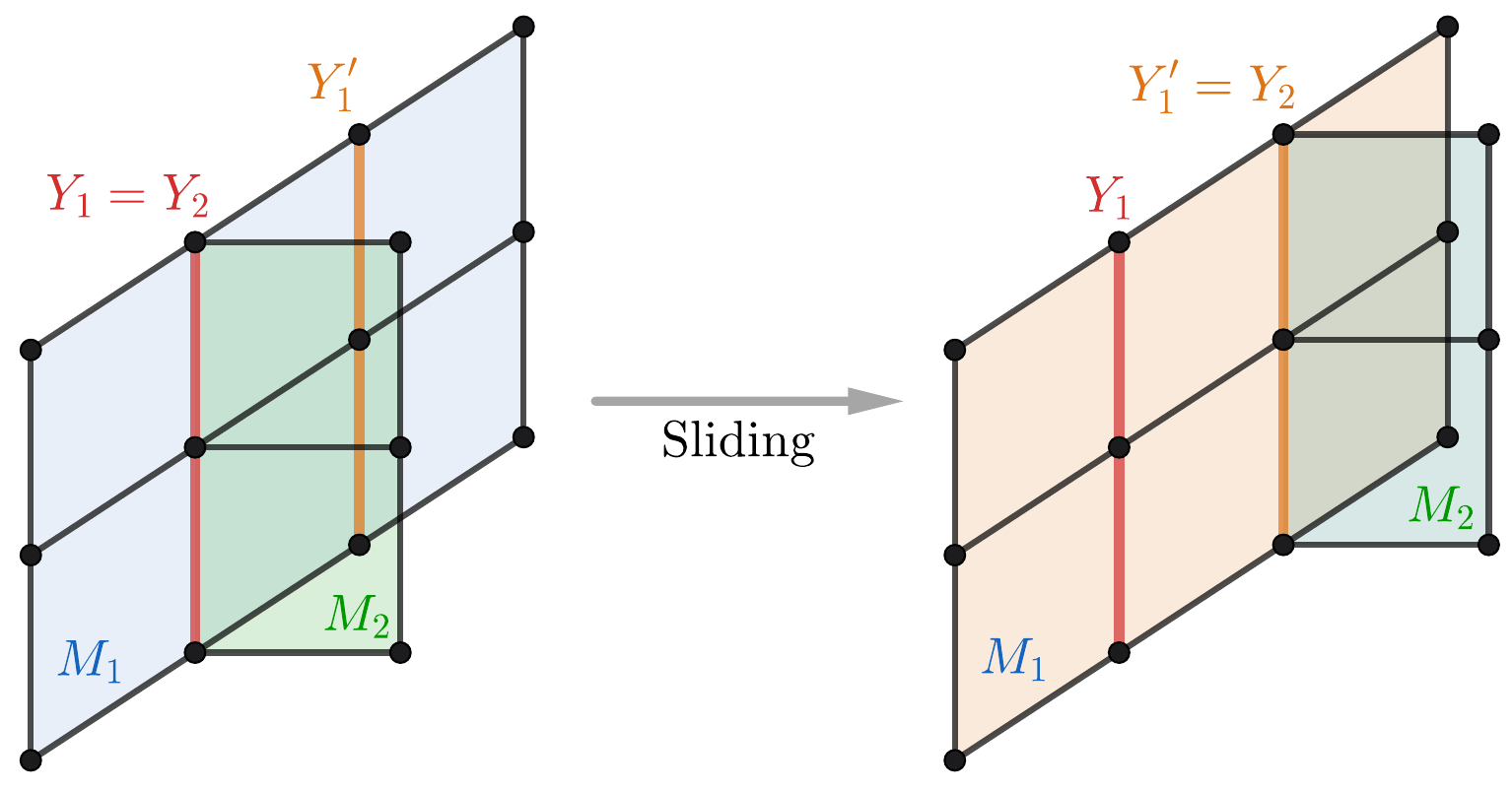}
\caption{Sliding in a median graph.}
\label{Sliding}
\end{center}
\end{figure}

\begin{definition}
Let $M_1, M_2$ be two median graphs, $Y_1 \leq M_1, Y_2 \leq M_2$ two convex subgraphs that we identify thanks to a graph-isomorphism $\varphi : Y_1 \to Y_2$, and $Y_1' \leq M_1$ a convex subgraph parallel to $Y_1$. Let $\pi : Y_1' \to Y_1$ denote the gate-projection. The median graph 
$$M':= M_1 \underset{Y_1'=_{\varphi \circ \pi} Y_2}{\ast} M_2$$
is obtained from 
$$M:= M_1 \underset{Y_1 =_\varphi Y_2}{\ast} M_2$$
by a \emph{sliding}. 
\end{definition}

\noindent
We conclude this section by proving that applying a sliding to a median graph does not change (up to isomorphism) its crossing graph. More precisely:

\begin{prop}\label{prop:MapSliding}
Let $M_1, M_2$ be two median graphs, $Y_1 \leq M_1, Y_2 \leq M_2$ two convex subgraphs that we identify thanks to a graph-isomorphism $\varphi : Y_1 \to Y_2$, and $Y_1' \leq M_1$ a convex subgraph parallel to $Y_1$. Let $\pi : Y_1' \to Y_1$ denote the gate-projection. The unique map 
$$\beta : M:=M_1 \underset{Y_1 =_\varphi Y_2}{\ast} M_2 \to M':= M_1 \underset{Y_1'=_{\varphi \circ \pi} Y_2}{\ast} M_2$$
that sends identically $(V(M_1),E(M_1))$ in $M$ to $(V(M_1),E(M_1))$ in $M'$ and similarly $(V(M_2) \backslash V(Y_2), E(M_2) \backslash E(M_2))$ in $M$ to $(V(M_2) \backslash V(Y_2), E(M_2) \backslash E(M_2))$ in $M'$ satisfies the following assertions:
\begin{itemize}
	\item $\beta$ preserves the $\Theta$-classes, i.e.\ $\beta(a) \Theta \beta(b)$ if and only if $a \Theta b$ for all $a,b \in E(M)$;
	\item moreover, for all $a,b \in E(M)$, the $\Theta$-classes $[\beta(a)]$ and $[\beta(b)]$ cross if and only if $[a]$ and $[b]$ cross. 
\end{itemize}
\end{prop}

\noindent
We emphasize that, typically, $\beta$ is not a graph-morphism. Indeed, if $x \in V(Y_2)$ and $y \in V(M_2) \backslash V(Y_2)$ are two adjacent vertices in $M$, then $\beta(x)$ and $\beta(y)$ are not adjacent if $Y_1 \neq Y_1'$. 

\begin{proof}[Proof of Propostion~\ref{prop:MapSliding}.]
Assume that $a,b \in E(M)$ belong to the same $\Theta$-class in $M$. Thus, there exists a sequence of edges $e_1:=a, e_2, \ldots, e_{n-1},e_n:=b$ such that $e_i$ and $e_{i+1}$ are opposite edges of some square for every $1 \leq i \leq n-1$. Let $1 \leq i(1) < i(2) < \cdots < i(k) \leq n-1$ denote the indices such that $\{ (e_{i(s)}, e_{i(s)+1}) \mid 1 \leq s \leq k\}$ are the successive pairs of edges such that one belongs to $E(M_1)$ but not the other. For every $1 \leq s \leq k$, the edge of $(e_{i(s)}, e_{i(s)+1})$ that does not belong to $E(M_1)$ is sent by $\beta$ to an edge of $M_2$ that is opposite in some square to an edge of $Y_1'$ in the same $\Theta$-class as the edge of $(e_{i(s)}, e_{i(s)+1})$ that belongs to $E(M_1)$. Thus, $\beta(e_{i(s)})$ and $\beta(e_{i(s)+1})$ belong to the same $\Theta$-class in $M'$. On the other hand, for every $r \notin \{i_1, \ldots, i_k\}$, the edges $\beta(e_r)$ and $\beta(e_{r+1})$ are opposite in some square of $M'$, so we conclude that $\beta(a)=\beta(e_1)$ and $\beta(b)= \beta(e_n)$ belong to the same $\Theta$-class in $M'$. 

\medskip \noindent
We have just proved that $\beta$ sends edges in a common $\Theta$-class to edges in a common $\Theta$-class. The same argument applies to $\beta^{-1}$ by symmetry. Thus, given two edges $a,b \in E(M)$, $\beta(a)$ and $\beta(b)$ belong to the same $\Theta$-class, then $\beta^{-1}(\beta(a))=a$ and $\beta^{-1}(\beta(b))=b$ also belong to the same $\Theta$-class in $M$. In other words, $\beta$ sends two initial edges to a common $\Theta$-class if and only if they are already belong to the same $\Theta$-class. This proves the first claim of our proposition.

\medskip \noindent
Next, let $a,b \in E(M)$ be two edges such that the $\Theta$-classes $[a]$ and $[b]$ cross. Up to replacing $a$ and $b$ with other representatives, we can assume that $a$ and $b$ have a common endpoint and span a square in $M$. The only case where $\beta(a)$ and $\beta(b)$ do not span a square in $M'$ is when, up to switching $a$ and $b$, $a \in E(Y_1)$ and $b \in E(M_2) \backslash E(Y_2)$. But, then, $\beta(b)$ spans a square in $M'$ with an edge $a' \in E(Y_1')$ in the same $\Theta$-class as $\beta(a)$ (namely, the gate-projection of $a$ to $Y_1'$). Therefore, the $\Theta$-classes $[\beta(a)]=[a']$ and $[\beta(b)]$ cross.

\medskip \noindent
We have just proved that $\beta$ sends crossing $\Theta$-classes to crossing $\Theta$-classes. Again, the same argument applies to $\beta^{-1}$ by symmetry. Thus, given two edges $a,b \in E(M)$, if $[\beta(a)]$ and $[\beta(b)]$ cross in $M'$, then $[\beta^{-1}(\beta(a))]= [a]$ and $[\beta^{-1}(\beta(b))]=[b]$ cross in $M$. We conclude, as desired, that $\beta$ sends two $\Theta$-classes to two crossing $\Theta$-classes if and only if they already cross in $M$. 
\end{proof}

\noindent
As an immediate consequence of Proposition~\ref{prop:MapSliding}, we get that:

\begin{cor}\label{cor:SameCrossingSliding}
Let $M$ and $M'$ be two median graphs. If $M'$ can be obtained from $M$ by a sliding, then the crossing graphs of $M$ and $M'$ are isomorphic. 
\end{cor}

\subsection{Proof of Theorem~\ref{thm:IntroMain}}

\noindent
This section is dedicated to the proof of the main result of this paper, namely Theorem~\ref{thm:IntroMain}, which we repeat for convinience:

\begin{thm}\label{thm:SameCrossing}
Two finite median graphs have isomorphic crossing graphs if and only if one can be obtained from the other by a sequence of slidings.
\end{thm}

\noindent
We start by description a particular family of slidings. In a median graph, given a hyperplane $J$, an \emph{orientation} $\vec{J}$ refers to an ordered pair $(J^+,J^-)$ where $J^\pm$ denote the halfspaces delimited by $J$. Accordingly, we use the notation $\partial_\pm J:= \partial J \cap J^\pm$ for the fibres of $J$. Since we know from Lemma~\ref{lem:FibreParallel} that the fibres of $J$ are parallel, it is possible to slide the halfspaces delimited by $J$ along $\vec{J}$. More formally:

\begin{definition}
Let $M$ be a median graph and $\vec{J}$ an oriented hyperplane. Our graph $M$ decomposes as 
$$(J^+ \cup N(J)) \underset{\partial_- J}{\ast} J^-.$$
The \emph{sliding of $M$ along $\vec{J}$} is
$$(J^+ \cup N(J)) \underset{\partial_+J =_\pi \partial_-J}{\ast} J^-$$
where $\pi : \partial_+J \to \partial_-J$ is the gate-projection. 
\end{definition}

\noindent
We are now ready to prove our theorem. 

\begin{proof}[Proof of Theorem~\ref{thm:SameCrossing}.]
In order to prove the theorem, it suffices to verify that, given a finite median graph $M$, one can always obtain the simplex graph of $\mathrm{Cross}(M)$ by applying a sequence of slidings to $M$. For this, we introduce the complexity
$$\kappa(M):= \sum\limits_{J_1,J_2 \text{ $\Theta$-classes}} d(N(J_1),N(J_2)).$$
If $\kappa(M)=0$, then the hyperplanes in $M$ are pairwise in contact, and it follows from Lemma~\ref{lem:PairwiseContact} that $M$ already coincides with $\mathrm{Simp}(\mathrm{Cross}(M))$, so there is nothing to do. Thus, from now on, we assume that $\kappa(M) \geq 1$. Our goal is to apply to $M$ a sliding along a well-chosen oriented hyperplane in order to create a new median graph $M'$ satisfying $\kappa(M')< \kappa(M)$. Iterating this construction yields a sequence of median graphs with smaller and smaller complexity, all with the same crossing graph according to Corollary~\ref{cor:SameCrossingSliding}. Since our complexity is an integer, such a sequence has to eventually stop, which happens precisely when $\mathrm{Simp}(\mathrm{Cross}(M))$ is reached, as desired. 

\medskip \noindent
Now, fix two $\Theta$-classes $J_1$ and $J_2$ not in contact (which exist since $\kappa(M)\geq 1$). According to Corollary~\ref{cor:SeparatingCarriers}, there exists some $\Theta$-class $K$ separating $N(J_1)$ and $N(J_2)$. Orient $K$ as $\vec{K}:= (K^+,K^-)$ where $K^+$ denotes the halfspace delimited by $K$ that contains $N(J_1)$. We claim that $\mathrm{Slid}(M,\vec{K})$ has a smaller $\kappa$-complexity than $M$. 

\medskip \noindent
Indeed, let $\beta : M \to \mathrm{Slid}(M,\vec{K})$ denote the map given by Proposition~\ref{prop:MapSliding}. Since $\beta$ induces a bijection from the set of $\Theta$-classes of $M$ onto the set of $\Theta$-classes of $\mathrm{Slid}(M,\vec{K})$, we have
$$\kappa(\mathrm{Slid}(M,\vec{K})) = \sum\limits_{H_1,H_2 \text{ $\Theta$-classes of } M} d( N(\beta(H_1)),N(\beta(H_2))).$$
Notice that, if $H_1$ and $H_2$ are two $\Theta$-classes of $M$ whose carriers are not separated by $K$, then $d(N(\beta(H_1)), N(\beta(H_2))) \leq d(N(H_1), N(H_2))$. Indeed, if $\alpha$ denotes a path of minimal length connecting $N(H_1)$ to $N(H_2)$ in $M$, then, due to the fact that $K$ does not separate $N(H_1)$ and $N(H_2)$, we know from Proposition~\ref{prop:HypSepConvex} that $K$ does not cross $\alpha$. Thus, $\alpha$ is contained in $K^-$ or $K^+$, which implies that $\beta(\alpha)$ yields a path of the same length as $\alpha$ connecting $N(\beta(H_1))$ to $N(\beta(H_2))$. This proves the desired inequality. Then, assume that $K$ separates $N(H_1)$ and $N(H_2)$, say $N(H_1) \subseteq K^+$ and $N(H_2) \subseteq K^-$. We claim that $d(N(\beta(H_1)),N(\beta(H_2)))< d(N(H_1),N(H_2))$. Since we know that there exist such pairs of $\Theta$-classes whose carriers are separated by $K$, namely $J_1$ and $J_2$, we will conclude that
$$\begin{array}{lcl} \displaystyle \kappa(\mathrm{Slid}(M,\vec{K})) & = & \displaystyle \sum\limits_{H_1,H_2 \text{ $\Theta$-classes of } M} d( N(\beta(H_1)),N(\beta(H_2))) \\ \\ & < & \displaystyle  \sum\limits_{H_1,H_2 \text{ $\Theta$-classes of } M} d(N(H_1),N(H_2)) = \kappa(M), \end{array}$$
as desired. As before, let $\alpha$ be a path of minimal length connecting $N(H_1)$ to $N(H_2)$. Because $K$ separates $N(H_1)$ and $N(H_2)$, $\alpha$ has an edge in $K$, say $\{a,b\}$ with $a \in K^+$ and $b \in K^-$. Notice that, if $\alpha_1$ denotes the initial subsegment of $\alpha$ up to $a$ and if $\alpha_2$ denotes the final subsegment of $\alpha$ starting from $b$ (i.e.\ $\alpha = \alpha_1 \cup \{a,b\} \cup \alpha_2$), then the final vertex of $\beta(\alpha_1)$ coincides with the initial vertex of $\beta(\alpha_2)$. Hence a path of smaller length connecting $N(\beta(H_1))$ to $N(\beta(H_2))$, proving the desired strict inequality. 
\end{proof}

\section{Graphs with extremal median degrees}

\subsection{Graphs with small median degrees}

\noindent
In this section, we characterise the connected graphs with median degree $\leq 3$.

\begin{prop}\label{prop:SmallMedianDegree}
Let $X$ be a connected graph. The following assertions hold:
\begin{itemize}
	\item $\mathrm{mdeg}(X)=1$ if and only if $X$ is a single vertex;
	\item $\mathrm{mdeg}(X)=2$ if and only if $X$ is single edge;
	\item $\mathrm{mdeg}(X)=3$ if and only if $X$ is a triangle $K_3$ or a star $K_1 \ast K_n$ with $n \geq 2$ arms.
\end{itemize}
\end{prop}

\noindent
The first two items will be essentially obvious, but the third item requires the following observation:

\begin{lemma}\label{lem:MedianDegreeThree}
A $1$-connected median graph $M$ has degree $3$ if and only if it is a $3$-cube $Q_3$ or a strip $K_2 \times P_n$ of length $n \geq 2$.
\end{lemma}

\begin{proof}
Because $M$ is $1$-connected, its cubical dimension has to be $\geq 2$. Indeed, otherwise $M$ would be a tree, but the only $1$-connected trees are $K_1$ and $K_2$, which have degree $\leq 2$. If $M$ has cubical dimension $\geq 3$, then it contains a $3$-cube, say $Q$. Since all the vertices in $Q$ have degree $3$, necessarily $M=Q$. From now on, we assume that $M$ has cubical dimension $2$.

\medskip \noindent
Because $M$ has degree $3$, it must contain some vertex of degree exactly $3$, say $x \in V(M)$. We know from Proposition~\ref{prop:CutVertexLink} that the cubical link of $x$ must be connected, so it is either a triangle or a path $P_2$. In the former case, the $3$-cube condition implies that $M$ contains a $3$-cube, contradicting our assumption that $M$ has cubical dimension $2$, so $\mathrm{link}_\square(x)$ is a path. Let $x_1,x_2,x_3 \in V(M)$ denote the three neighbours of $x$, indexed such that the edge $\{x,x_2\}$ spans a square with both $\{x,x_1\}$ and $\{x,x_3\}$. Let $J$ denote the $\Theta$-class of $\{x,x_2\}$. Each fibre of $J$ is a median graph of degree $\leq 2$ and of cubical dimension $\leq \dim_\square(M)-1=1$, so it has to be a path $P_n$. Notice that $n \geq 2$ since the edges $\{x,x_1\}$ and $\{x,x_3\}$ belong to a common fibre of $J$. Thus, the carrier of $J$ is isomorphic to the strip $K_2 \times P_n$. 

\medskip \noindent
We claim that $M=N(J)$, which will conclude the proof of our lemma. 

\medskip \noindent
\begin{minipage}{0.48\linewidth}
\begin{center}
\includegraphics[width=0.98\linewidth]{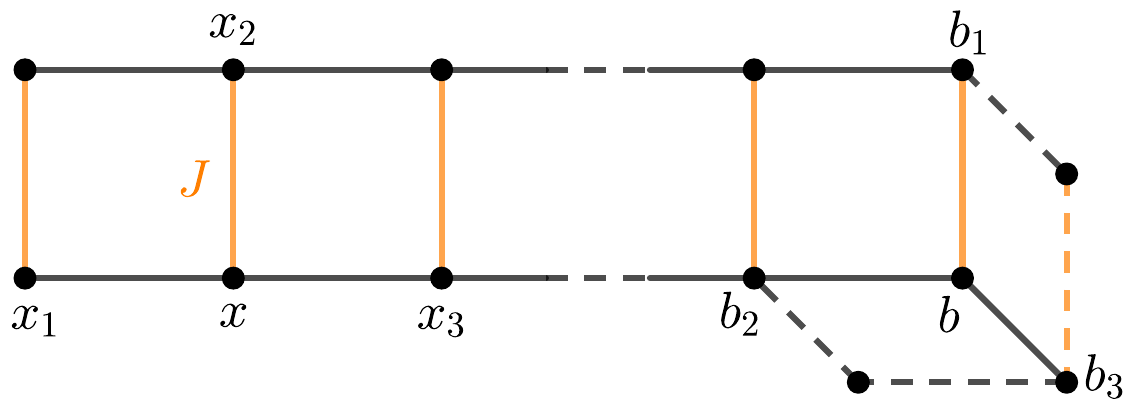}
\end{center}
\end{minipage}
\begin{minipage}{0.5\linewidth}
If not, there exists $a \in V(M)$ adjacent to a vertex $b \in V(N(J))$ but not in $N(J)$. Because $M$ has degree $3$, $b$ must be one of the four corners of the strip $N(J)$. Let $b_1,b_2 \in V(N(J))$ denote the two neighbours of $b$ in $N(J)$, indexed such that $\{b,b_1\} \in J$. Because the cubical link of $b$ must be connected, again according to Proposition~\ref{prop:CutVertexLink}, the edge $\{b,a\}$ must span a square with either $\{b,b_1\}$ or $\{b,b_2\}$. 
\end{minipage}

\medskip \noindent
The former case is not possible, since otherwise $a$ would belong to $N(J)$. But the latter case is not possible either, since otherwise the fourth vertex of the square spanned by $\{b,a\}$ and $\{b,b_2\}$ would provide a fourth neighbour of $b_2$, contradicting the fact that $M$ has degree $3$. 
\end{proof}

\begin{proof}[Proof of Proposition~\ref{prop:SmallMedianDegree}.]
Since the only median graph of degree $1$ (resp.\ $2$) is a single edge (resp.\ square), whose crossing graph is a single vertex (resp.\ edge), the first two items of our proposition follow immediately. The third item follows from the characterisation provided by Lemma~\ref{lem:MedianDegreeThree} of the $1$-connected median graphs of degree $3$ and from the equivalence provided by Proposition~\ref{prop:ConnectedCrossing} between being $1$-connected and having a connected crossing graph. 
\end{proof}

\subsection{Graphs with maximal median degree}

\noindent
In this section, we prove the second main result of this article, namely we characterise exactly when the median degree of a graph $X$ is as large as possible, namely $|V(X)|$. 

\begin{thm}
Let $X$ be a finite graph. The following assertions are equivalent:
\begin{itemize}
	\item[(i)] the equality $\mathrm{mdeg}(X)= |V(X)|$ holds;
	\item[(ii)] there exists a unique median graph whose crossing graph is $X$, namely $\mathrm{Simp}(X)$;
	\item[(iii)]  $X$ has no separating star. 
\end{itemize}
\end{thm}

\begin{proof}
If there exist at least two median graphs having $X$ as their crossing graphs, in particular there exists at least one such median graph $M$ that is distinct from $\mathrm{Simp}(X)$. According to Lemma~\ref{lem:PairwiseContact}, $M$ has two $\Theta$-classes $J_1$ and $J_2$ not in contact. We know from Corollary~\ref{cor:SeparatingCarriers} that the carriers of $J_1$ and $J_2$ are separated by some $\Theta$-class $J$. Clearly, along any path connecting $J_1$ to $J_2$ in $\mathrm{Cross}(M)$, one $\Theta$-class must coincide with or cross $J$. In other words, the star of $J$ separates $J_1$ and $J_2$ in $\mathrm{Cross}(M)$, proving that $X$ contains a separating star. Thus, the implication $(iii) \Rightarrow (ii)$ holds.

\medskip \noindent
If there exists a unique median graph whose crossing graph is $X$, this must be $\mathrm{Simp}(X)$, hence
$$\mathrm{mdeg}(X)= \mathrm{deg}(\mathrm{Simp}(X)) = |V(X)|.$$
This proves the implication $(ii) \Rightarrow (i)$. 

\medskip \noindent
Finally, assume that $X$ contains a separating star, say $\mathrm{star}(x)$ for some $x \in V(X)$. Fix a decomposition $V(X)= V(A) \sqcup V(\mathrm{star}(x)) \sqcup V(B)$ where $A$ and $B$ are two non-empty induced subgraphs separated by $\mathrm{star}(x)$ in $X$. 

\medskip \noindent
\begin{minipage}{0.38\linewidth}
\includegraphics[width=0.98\linewidth]{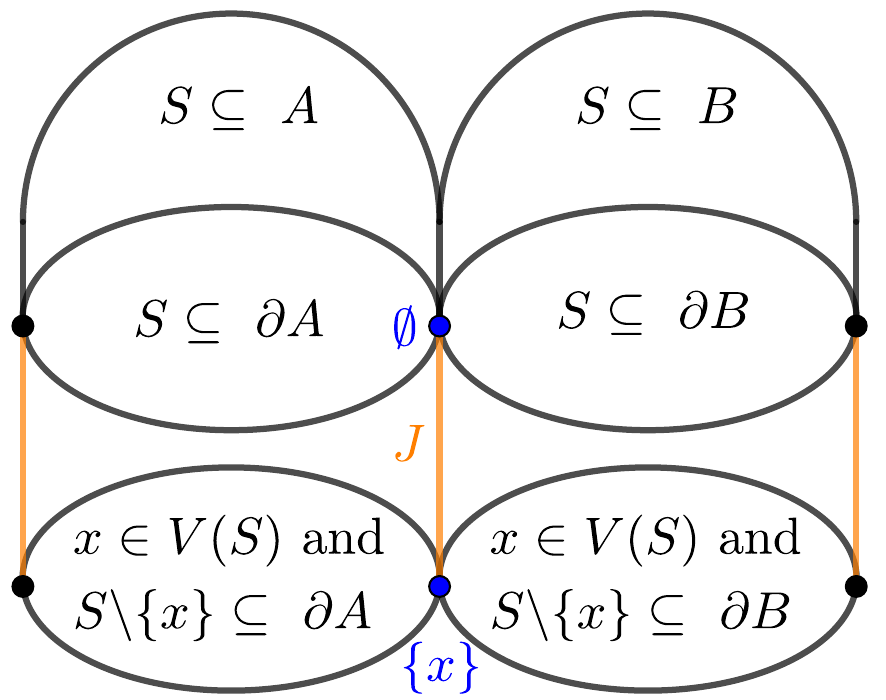}
\end{minipage}
\begin{minipage}{0.6\linewidth}
Then, $\mathrm{Simp}(X)$ can be decomposed as shown by the figure on the left, where $J$ denotes the $\Theta$-class of $\{\emptyset, \{x\}\}$ and where $\partial A$ (resp.\ $\partial B$) denotes the subgraph induced by the vertices not in $A$ (resp.\ $B$) but adjacent to vertices in $A$ (resp.\ $B$). Notice that
$$\{ S \subseteq \partial B\} \text{ and } \{S \mid x \in V(S) \text{ and } S \backslash \{x\} \subseteq \partial B \}$$
are parallel in $\mathrm{Simp}(X)$ as a consequence of Lemmas~\ref{lem:SimpHyp} and~\ref{lem:FibreParallel}. 
\end{minipage}

\medskip \noindent
From the decomposition
$$\mathrm{Simp}(X)= \{ S \mid V(S) \subseteq V(B \cup \partial B) \} \underset{\{S \subseteq \partial B\}}{\ast} \{ S \mid V(S) \nsubseteq V(B)\},$$
we can construct a new median graph $M$ by sliding, namely
$$M:=\{ S \mid V(S) \subseteq V(B \cup \partial B) \} \underset{\{S \subseteq \partial B\} = \{ S \mid x \in V(S) \text{ and } S\backslash \{x\} \subseteq \partial B\}}{\ast} \{ S \mid V(S) \nsubseteq V(B)\}$$
where we use the gluing map
$$\left\{ \begin{array}{ccc} \{S \subseteq \partial B\} & \to &  \{ S \mid x \in V(S) \text{ and } S\backslash \{x\} \subseteq \partial B\} \\ S & \mapsto & S \cup \{x\} \end{array} \right..$$
We claim that $\mathrm{deg}(M)< |V(X)|$, which will prove the implication $(i) \Rightarrow (iii)$ and conclude the proof of our theorem. 

\medskip \noindent
\begin{minipage}{0.38\linewidth}
\includegraphics[width=0.98\linewidth]{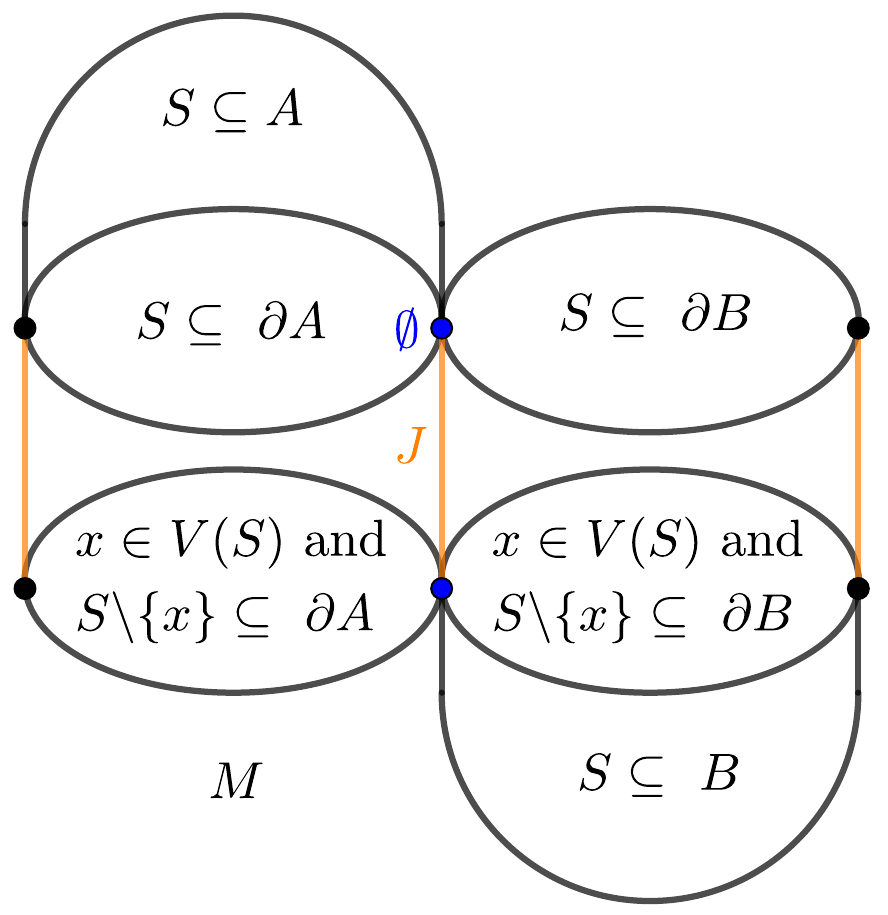}
\end{minipage}
\begin{minipage}{0.6\linewidth}
When passing from $\mathrm{Simp}(X)$ to $M$, the degrees of the vertices in $\{S \subseteq \partial B\}$ do not increase, the degrees of the vertices in $\{ S \mid x \in V(S) \text{ and } S \backslash \{x\} \subseteq \partial B \}$ may increase, and the degrees of the other vertices do not change. Since $\emptyset$ is the only vertex of degree $|V(X)|$ in $\mathrm{Simp}(X)$, it suffices to verify that $\emptyset$ has degree $<|V(X)|$ in $M$ and also that the vertices in $\{ S \mid x \in V(S) \text{ and } S \backslash \{x\} \subseteq \partial B \}$ have degrees $<|V(X)|$ in $M$. The former assertion is clear since, given an arbitrary vertex $b \in V(B)$, $\emptyset$ is adjacent to the singleton $\{b\}$ in $\mathrm{Simp}(X)$ but not longer in $M$. 
\end{minipage}

\medskip \noindent
Now, let $S$ be a clique of $X$ satisfying $x \in V(S)$ and $S \backslash \{x\} \subseteq \partial B$. Its neighbours in $M$ are its neighbours in $\mathrm{Simp}(X)$ and the neighbours of $S \backslash \{x\}$ in $\mathrm{Simp}(X)$ contained in $\{S \subseteq \partial B\}$. In other words, the neighbours of $S$ in $M$ are: the cliques of $X$ of the form $S \cup \{y\}$ and $S \backslash \{y\}$; the cliques of $X$ of the form $(S\backslash \{x\}) \cup \{y\}$ and $(S \backslash \{x,y\})$ that are contained in $\partial B$. Notice that, in all the cases, the vertex $y$ that we add or remove always belongs to $S \cup \partial B \cup \mathrm{link}(x) \subseteq \mathrm{star}(x)$. Therefore, the degree of $S$ in $M$ is $\leq |\mathrm{star}(x)| < |V(X)|$, where the strict inequality follows from the fact that $A$ and $B$ are non-empty subgraphs disjoint from $\mathrm{star}(x)$. 

\medskip \noindent
Thus, we have proved that $M$ has smaller degree than $\mathrm{Simp}(X)$, as desired. 
\end{proof}

\noindent
It is worth mentioning that, even when only one sliding can be performed on the simplex-graph $\mathrm{Simp}(X)$ of some graph $X$, during the process the degree can decrease by an arbitrarily large amount. In other words, satisfying the property that $|\mathrm{Cross}^{-1}(X)|=2$ does not impose any restriction on the median degree compared to the number of vertices $|V(X)|$. Our observation is justified by the following example. 

\begin{ex}
Given an $n \geq 3$, let $A_n$ denote the graph obtained by connected two copies of the complete graph $K_n$ with a path of length $2$. Its simplex-graph $\mathrm{Simp}(A_n)$ can be described as follows. Let $L_4$ denote the graph obtained by gluing four squares $S_1, \ldots, S_4$ such that $S_1 \cap \cdots \cap S_4$ is a single vertex, such that $S_i \cap S_{i+1}$ is a single edge for every $1 \leq i \leq 3$ and such that $S_i \cap S_{i-1}$, $S_i \cap S_{i+1}$ are two consecutive edges along $S_i$ for every $2 \leq i \leq 3$. Then, $\mathrm{Simp}(A_n)$ can be obtained from $L_4$ by gluing two $n$-cubes on the squares $S_1$ and $S_4$. For $n=3$, we have:
\begin{center}
\includegraphics[width=0.7\linewidth]{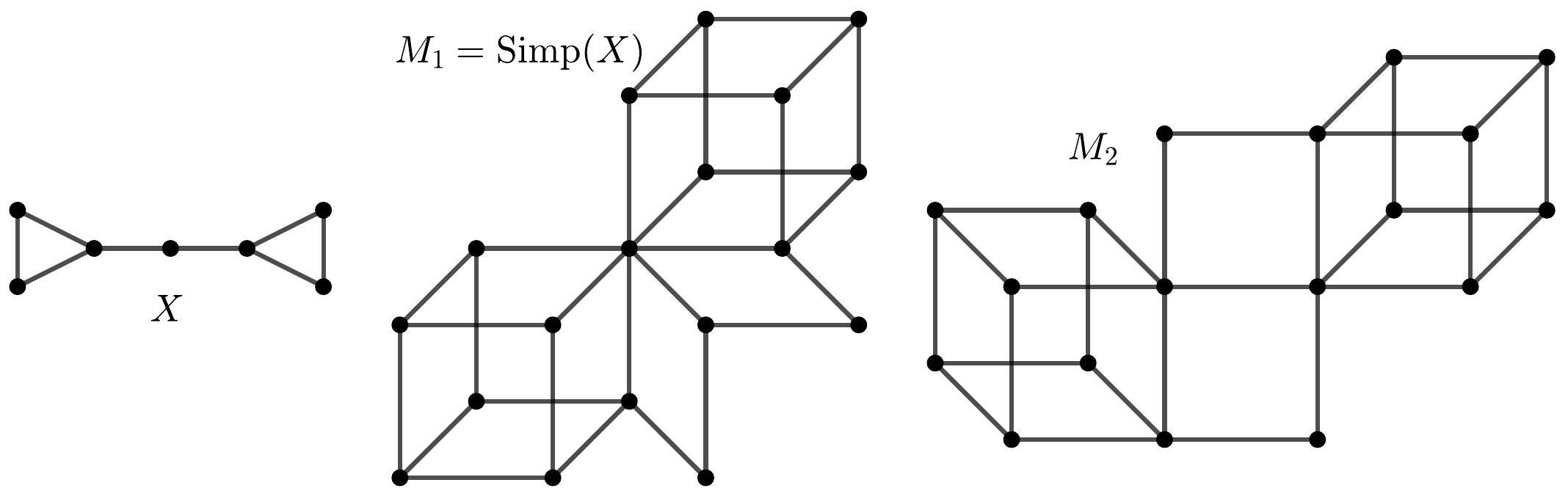}
\end{center}
By checking all the possible slidings that can be applied, we find that exactly two median graphs $M_1:= \mathrm{Simp}(A_n)$ and $M_2$ have $A_n$ as their crossing graph. In order to describe $M_2$, let $Z_4$ denote the graph obtained by gluing four squares $S_1, \ldots, S_4$ such that $S_1 \cap \cdots \cap S_4$ is empty, such that $S_i \cap S_{i+1}$ is a single edge for every $1 \leq i \leq 3$ and such that $S_i \cap S_{i-1}$, $S_i \cap S_{i+1}$ are two consecutive edges along $S_i$ for every $2 \leq i \leq 3$. Then, $M_2$ can be obtained from $Z_4$ by gluing two $n$-cubes on the squares $S_1$ and $S_4$. We conclude that $|\mathrm{Cross}^{-1}(A_n)|=2$ but $\mathrm{mdeg}(A_n)= n+2$ (compared to $|V(A_n)|= 2n+1$). 
\end{ex}

\section{Examples of median degrees}\label{section:Example}

\noindent
In this section, we record some computations of median degrees for a few explicit examples of graphs. We start with grids:

 \begin{prop}\label{prop:Grid}
Let $n \geq m \geq 0$ be two integers. 
$$\mathrm{mdeg}([0,n] \times [0,m]) = \left\{ \begin{array}{cl} 1 & \text{if } m=n=0 \\ 2 & \text{if } (m,n)=(0,1) \\ 3 & \text{if } (m,n)=(0,2) \\ 4 & \text{if } m=0 \text{ and } n \geq 3 \\ 4 & \text{if } m=1 \text{ and } n \geq 1 \\ 7 & \text{if } m=n=2 \\ 8 & \text{if } m=2 \text{ and } n \geq 3 \\ mn-4 & \text{if } m,n  \geq 3 \end{array} \right..$$
\end{prop}

\noindent
In order to prove the proposition, the following couple of inequalities will be needed:

\begin{lemma}\label{lem:NotStarSeparated}
Given a graph $X$, 
$$\mathrm{mdeg}(X) \geq \max \{ |V(Y)| \mid Y \leq X \text{ finite not separated by a star in } X\}.$$
\end{lemma}

\begin{proof}
Fix a finite subgraph $Y \leq X$ such that none of its vertices are separated by some star in $X$. Let $M$ be a median graph of degree $\mathrm{mdeg}(X)$ with crossing graph isomorphic to $X$. The vertices of the finite subgraph $Y$ yield finitely many $\Theta$-classes in $M$, say $J_1, \ldots, J_n$ with $n:= |V(Y)|$. If these $\Theta$-classes are pairwise in contact, then it follows from Corollary~\ref{cor:ContactDegree} that
$$\mathrm{mdeg}(X) \geq \mathrm{deg}(M) \geq n = |V(Y)|,$$
as desired. Now, assume that $J_1, \ldots, J_n$ are not pairwise in contact, say $J_1$ and $J_2$ are not in contact. According to Corollary~\ref{cor:SeparatingCarriers}, there must exist some $\Theta$-class $J$ that separates the carriers of $J_1$ and $J_2$. Our $\Theta$-classes $J,J_1,J_2$ correspond to three vertices $x,x_1,x_2$ in $X$. We claim that the star of $x$ separates $x_1$ and $x_2$. Indeed, a path in $X$ connecting $x_1$ to $x_2$ yields a sequence of successively crossing $\Theta$-classes connecting $J_1$ and $J_2$. Because $J$ separates the carriers of $J_1$ and $J_2$, necessarily at least one $\Theta$-class along this sequence must coincide with or cross $J$. In other words, our path in $X$ connecting $x_1$ to $x_2$ has to pass through the star of $x$. In other words, we have proved the following observation, which we record for future use:

\begin{fact}\label{fact:SepStar}
If two $\Theta$-classes $A$ and $B$ are not in contact, then the vertices they represent in the crossing graph are separated by the star of any vertex representing a $\Theta$-class separating $A$ and $B$.
\end{fact}

\noindent
Since we have assumed that no star of $X$ separates two vertices in $Y$, this second case cannot occur. 
\end{proof}


\begin{lemma}\label{lem:UniqueStar}
Let $X$ be a locally finite graph. Assume that there exists a unique vertex $x \in V(X)$ whose star separates $X$. Then,
$$\mathrm{mdeg}(X) \geq |\mathrm{star}(x)| + \frac{1}{2} \# \{\text{connected components of } X \backslash \mathrm{star}(x)\}.$$
\end{lemma}

\begin{proof}
Fix vertices $x_1, \ldots, x_k \in V(X)$, one in each connected component of $X \backslash \mathrm{star}(x)$. Let $M$ be a median graph of degree $\mathrm{mdeg}(X)$ whose crossing graph is isomorphic to $X$. To the vertices $x,x_1, \ldots, x_k$ in $X$ correspond $\Theta$-classes $J, J_1, \ldots, J_k$ in $M$; and, to the neigbours of $x$ in $X$ correspond $\Theta$-classes $K_1, \ldots, K_s$ in $M$. The $\Theta$-classes $J_1, \ldots, J_k$ are distinct from and do not cross $J$, since $x_1, \ldots, x_k$ do not belong to the star of $x$. Therefore, each $J_i$ is contained in one of the two halfspaces delimited by $J$. Up to reindexing our $\Theta$-classes, say that $J_1, \ldots, J_r$ are all contained in the same halfspace $J^+$ delimited by $J$ for some $r \geq k/2$. Our goal is to justify that $J,J_1, \ldots, J_r, K_1, \ldots, K_s$ are pairwise in contact. This will prove that $M$ contains a vertex of degree $\geq r+s+1$, hence
$$\mathrm{mdeg}(X) \geq \mathrm{deg}(M) \geq r+s+1 \geq |\mathrm{star}(x)| +k/2,$$
as desired.

\medskip \noindent
Assume that $H$ is a $\Theta$-class not in contact with $J_i$ for some $1 \leq i \leq r$. According to Fact~\ref{fact:SepStar}, the two vertices represented by $H$ and $J_i$ are separated by the star of some vertex representing a $\Theta$-class separating $H$ and $J_i$. By assumption, this star has to be $\mathrm{star}(x)$, so $J$ has to separate $H$ and $J_i$. As a consequence, $H$ must be distinct from $J_1, \ldots, J_r, J, K_1, \ldots, K_s$. Thus, we have proved that each $J_i$ is in contact with all of $J, J_1, \ldots, J_r, K_1, \ldots, K_s$. Then, we deduce from Fact~\ref{fact:SepStar} that $K_1, \ldots, K_s$ are pairwise in contact, since the vertices they represent cannot be separated by a star (which would have to be $\mathrm{star}(x)$). We also know that $J$ crosses, and a fortiori is in contact with, $K_1, \ldots, K_n$. We conclude, as desired, that the $\Theta$-classes $J, J_1, \ldots, J_r, K_1, \ldots, K_s$ are pairwise in contact. 
\end{proof}

\begin{figure}
\begin{center}
\includegraphics[width=0.7\linewidth]{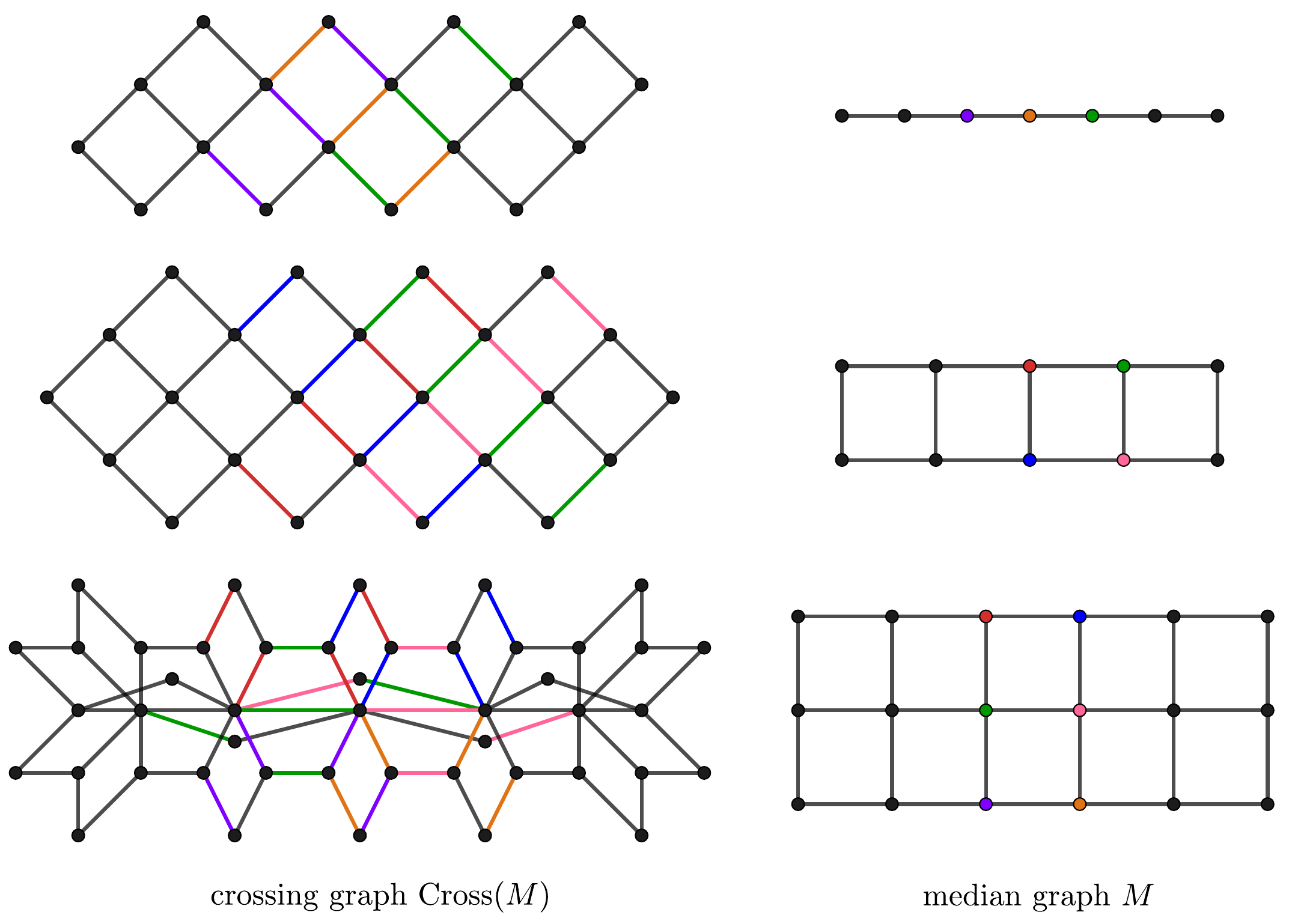}
\caption{Some median graphs with grids as crossing graphs.}
\label{GridCrossing}
\end{center}
\end{figure}
\begin{proof}[Proof of Proposition~\ref{prop:Grid}.]
If $m=n=0$, or $(m,n)=(0,1)$, or $(m,n)=(0,2)$, the value of the median degree of our grid follows from Proposition~\ref{prop:SmallMedianDegree}. If $m=0$ and $n \geq 3$, then our grid is a path $P_n$ of length $n$. We know from Proposition~\ref{prop:SmallMedianDegree} that its median degree is $\geq 4$, so it suffices to find a median graph of degree $4$ whose crossing graph is $P_n$. This is done by Figure~\ref{GridCrossing}. If $m=1$ and $n \geq 1$, then our grid is a ladder $K_2 \times P_n$ of length $n$. Again, we know from Proposition~\ref{prop:SmallMedianDegree} that its median degree is $\geq 4$, so it suffices to find a median graph of degree $4$ whose crossing graph is $P_n$, which is done by Figure~\ref{GridCrossing}. 

\medskip \noindent
Now, assume that $m=n=2$. Let $c_0$ denote the central vertex of our grid and $c_1, \ldots, c_4$ its four corners. Notice that $c_0$ is the only vertex whose star is separating; moreover, its star separates the grid into four connected components, namely $\{c_1\},\{c_2\},\{c_3\}, \{c_4\}$. Thus, Lemma~\ref{lem:UniqueStar} implies that our median degree is $\geq |\mathrm{star}(c_0)| + 4/2 = 5+ 2=7$. Figure~\ref{GridCrossing} provides a median graph of degree $7$ whose crossing graph is the $(2 \times 2)$-grid, so we conclude that the median degree of our grid is precisely $7$.

\medskip \noindent
\begin{minipage}{0.43\linewidth}
\includegraphics[width=0.95\linewidth]{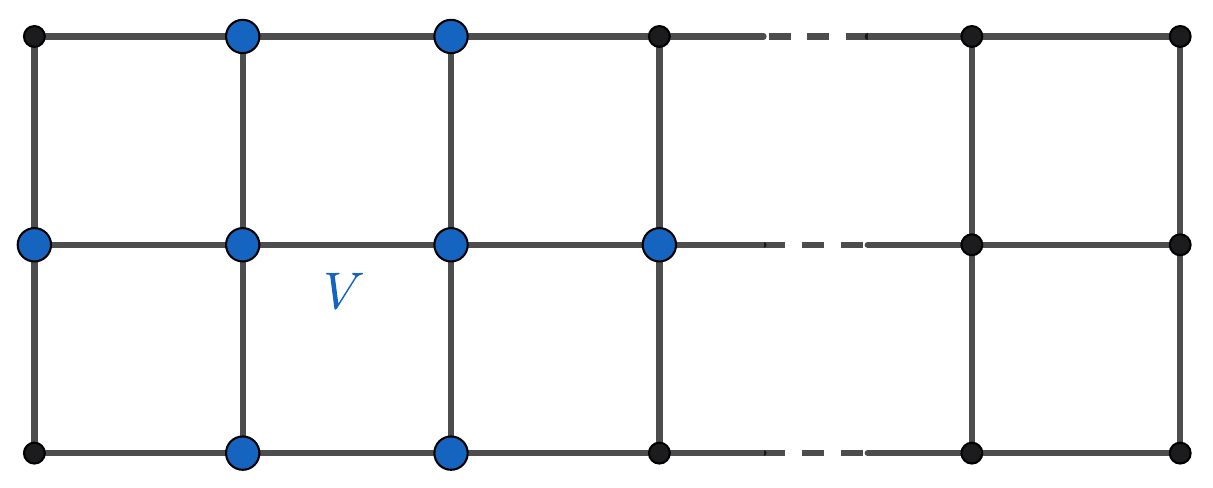}
\end{minipage}
\begin{minipage}{0.55\linewidth}
Next, assume that $m=2$ and $n \geq 3$. Let $V$ denote the set of vertices as illustrated by the figure on the left. Notice that no star of our grid separates two vertices of $V$, so Lemma~\ref{lem:NotStarSeparated} implies that the median degree of the grid is $\geq |V|=8$. Again, Figure~\ref{GridCrossing} provides a median graph of degree $8$ whose crossing graph is the $(n \times 2)$-grid, so we cocnlude that the median degree of our grid is precisely $8$.
\end{minipage}

\medskip \noindent
Finally, assume that $m,n \geq 3$. Let $C$ denote the graph obtained from our $(m \times n)$-grid by removing the four corners. Since no two vertices of $C$ are separated by a star in the full grid, Lemma~\ref{lem:NotStarSeparated} implies that the median degree we are looking for is $\geq mn-4$. It remains to construct a median graph of degree $mn-4$ with an $(m \times n)$-grid as its crossing graph. For this, we follow the following construction:
\begin{center}
\includegraphics[width=0.6\linewidth]{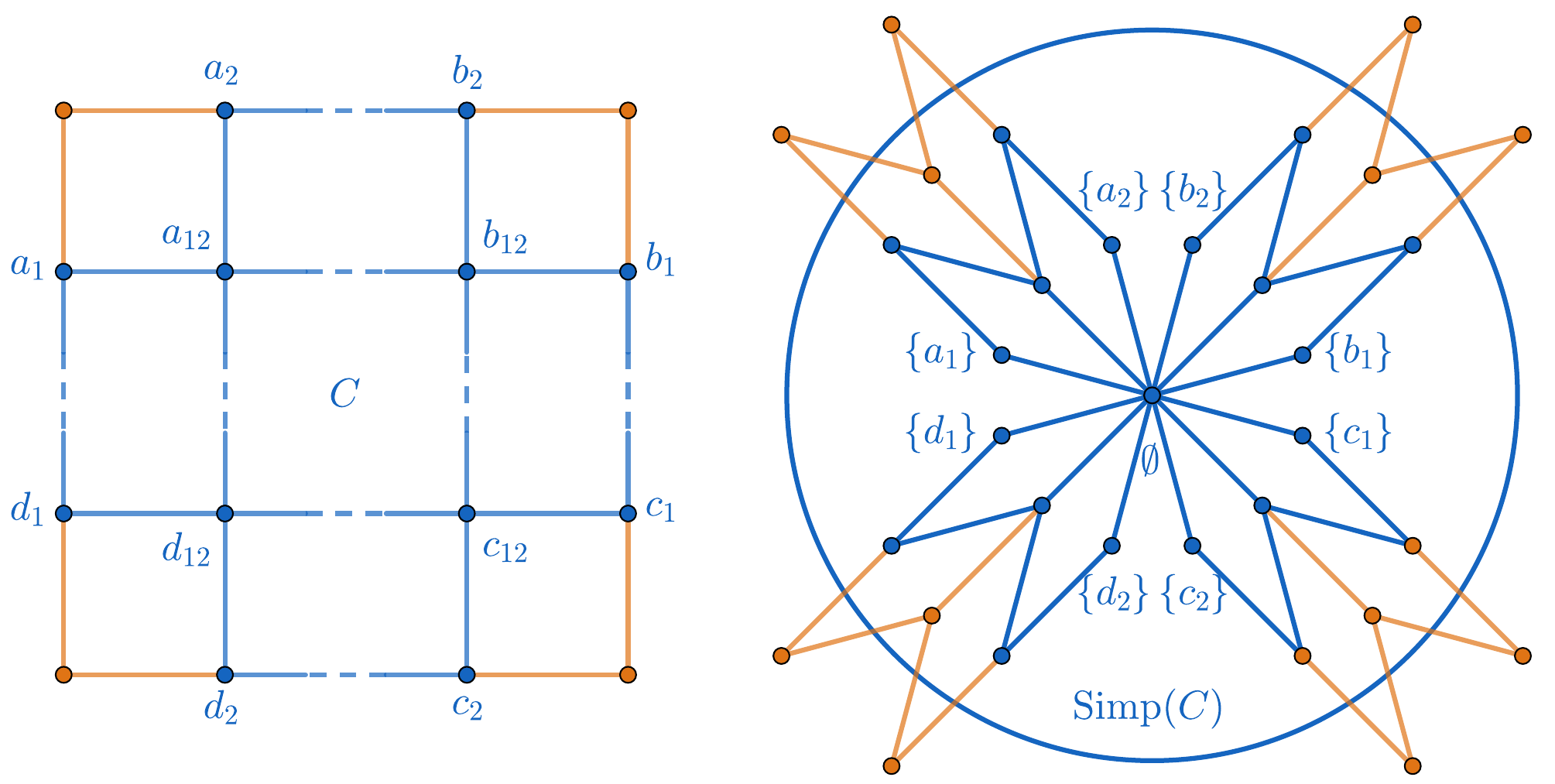}
\end{center}
More precisely, we start with the simplex-graph $\mathrm{Simp}(C)$. Its crossing graph is $C$ and it has degree $|C|=mn-4$, so it remains to add four new $\Theta$-classes corresponding to the four corners, but without increasing the degree. For this, we glue a $(2x1)$-ladder along the path $\{a_1,a_{12}\}, \{a_{12}\}, \{a_2,a_{12}\}$. This increases by one the degrees of $a_1$, $a_{12}$, and $a_2$, but we know that $\emptyset$ is the only vertex of maximal degree $mn-4$, so globally this does not modify the degree of our graph. Also, notice that the graph thus obtained is still median because we are amalgamated two median graphs along a convex path. Indeed, saying that the path $\{a_1,a_{12}\}, \{a_{12}\}, \{a_2,a_{12}\}$ is not convex amounts to saying that the edges $\{ \{a_{12}\}, \{a_1,a_{12}\}\}$ and $\{\{a_{12}\}, \{a_2,a_{12}\}\}$ span a square, or equivalently that $a_1$ and $a_2$ are two adjacent vertices in our grid, which is not the case. We perform a similar operation for $\{b_1,b_{12}\}, \{b_{12}\}, \{b_2,b_{12}\}$, for $\{c_1,c_{12}\}, \{c_{12}\}, \{c_2,c_{12}\}$, and for $\{d_1,d_{12}\}, \{d_{12}\}, \{d_2,d_{12}\}$. In the end, we get the desired the desired median graph. 
\end{proof}

\noindent
Notice that, as an immediate byproduct of Lemma~\ref{lem:NotStarSeparated}, we get the following estimate:

\begin{cor}\label{cor:CliqueGirth}
For every graph $X$,
$$\mathrm{mdeg}(X) \geq \mathrm{clique}(X), \mathrm{girth}(X).$$
As a consequence, for every $n \geq 1$, the complete graph $K_n$ with $n$ vertices and the cycle $C_n$ of length $n$ have median degree $n$.  
\end{cor}

\noindent
Lemma~\ref{lem:UniqueStar} also has further applications, for instance:

\begin{cor}
If $Q_3^-$ denotes the graph obtained from the $3$-cube $Q_3$ by removing a vertex, then $\mathrm{mdeg}(Q_3^-)=6$. 
\end{cor}

\begin{proof}
Let $x_0$ denote the central vertex of $Q_3^-$ and $x_1,x_2,x_3$ the three vertices that do not belong to $\mathrm{star}(x_0)$. Notice that $x_0$ is the only vertex whose star is separation; moreover, its star separates the graph into three connected components. Thus, Lemma~\ref{lem:UniqueStar} implies that $\mathrm{mdeg}(Q_3^-) \geq |\mathrm{star}(x_0)| + 3/2$, hence $\mathrm{mdeg}(Q_3^-) \geq 6$. In order to prove the reverse inequality, notice that Figure~\ref{CubeMinus} provides a median graph of degree $6$ having $Q_3^-$ as its crossing graph.
\end{proof}
\begin{figure}
\begin{center}
\includegraphics[width=0.5\linewidth]{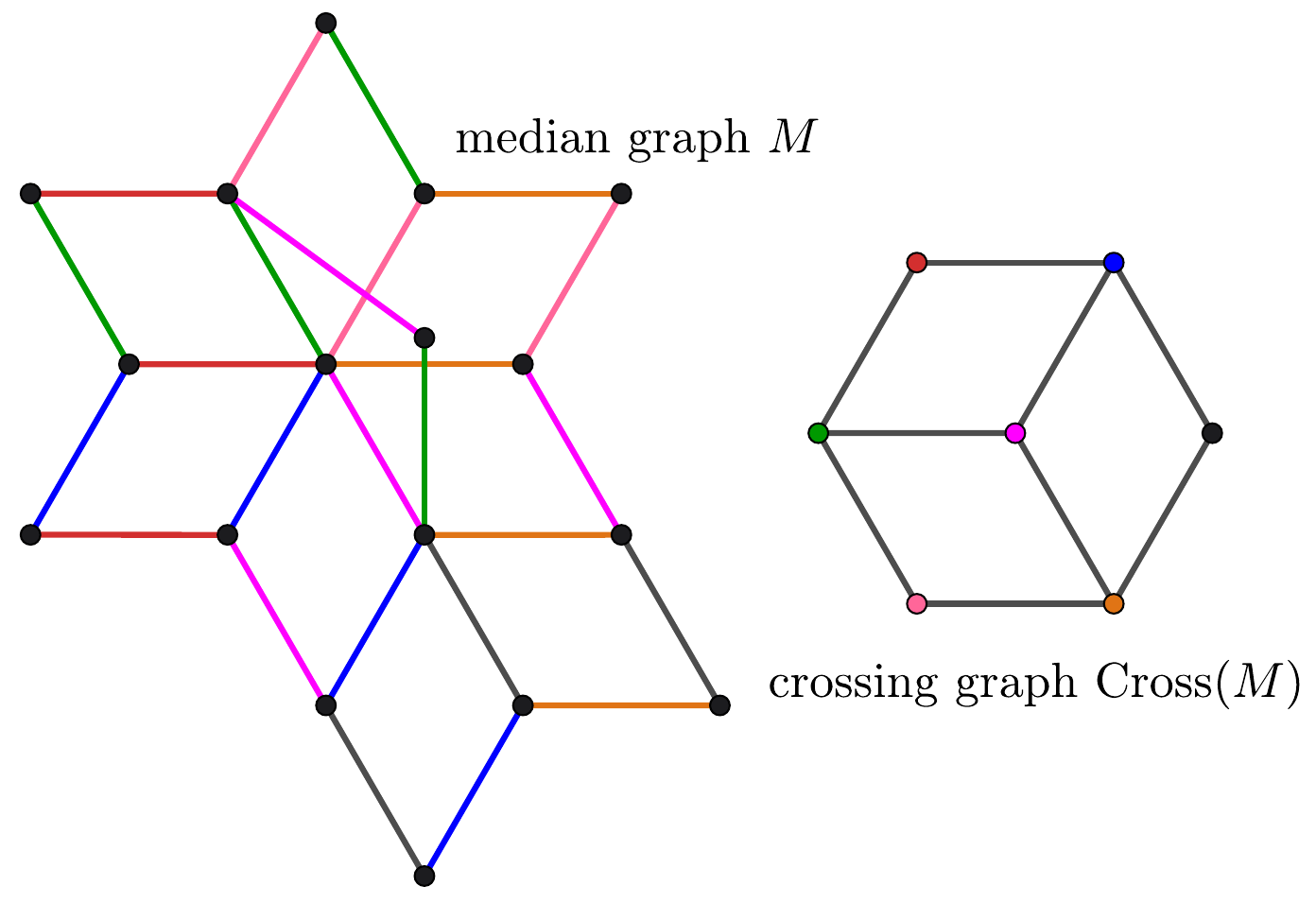}
\caption{A median graph of degree $6$ with $Q_3^-$ as its crossing graph.}
\label{CubeMinus}
\end{center}
\end{figure}

\noindent
It would be also interesting to understand how the median degree behaves under classical graph-theoretic products. A first example is given by \emph{joins}. Recall that the \emph{join} $A \ast B$ of two graphs $A$ and $B$ is the graph obtained from the distjoint union of $A$ and $B$ by adding an edge from every vertex of $A$ to every vertex of $B$. 

\begin{prop}\label{prop:Join}
For all finite graphs $X$ and $Y$,
$$\mathrm{mdeg}(X \ast Y) = \mathrm{mdeg}(X) + \mathrm{mdeg}(Y).$$
\end{prop}

\begin{proof}
If $M$ (resp.\ $N$) is a median graph of degree $\mathrm{mdeg}(X)$ (resp.\ $\mathrm{mdeg}(Y)$), then the product $M \times N$ is a median graph of degree $\mathrm{mdeg}(X)+ \mathrm{mdeg}(Y)$ whose crossing graph is isomorphic to $X \ast Y$. Hence 
$$\mathrm{mdeg}(X \ast Y) \leq \mathrm{mdeg}(X) + \mathrm{mdeg}(Y).$$
In order to prove the reverse inequality, let $Q$ be a median graph of degree $\mathrm{mdeg}(X \ast Y)$ whose crossing graph is isomorphic to $X \ast Y$. According to \cite[Theorem~9]{MR2299691} (see also \cite[Lemma~2.5]{MR2827012}), $Q$ decomposes as a product $Q_1 \times Q_2$ such that the crossing graph of $Q_1$ (resp.\ $Q_2$) coincides with $X$ (resp.\ $Y$). Therefore, 
$$\mathrm{mdeg}(X \ast Y)= \mathrm{deg}(Q) = \mathrm{deg}(Q_1)+ \mathrm{deg}(Q_2) \geq \mathrm{mdeg}(X) + \mathrm{mdeg}(Y).$$
This concludes the proof of our proposition.
\end{proof}

\noindent
As an application, we can mention:

\begin{cor}\label{cor:BiComplete}
For all $m \geq n \geq 1$, if $K_{n,m}$ denotes the complete bipartite graph on $n+m$ vertices, then  
$$\mathrm{mdeg}(K_{n,m}) = \left\{ \begin{array}{cl} 2 & \text{if } m=n=1 \\ 3 & \text{if } m \geq 2 \text{ and } n=1 \\ 4 & \text{if } m,n \geq 2 \end{array} \right..$$
\end{cor}

\begin{proof}
According to Proposition~\ref{prop:Join}, it suffices to verify that:

\begin{claim}
Given an $r \geq 1$, let $\overline{K_r}$ denote the graph that has $r$ vertices but no edges. Then
$$\mathrm{mdeg} \left( \overline{K_r} \right)= \left\{ \begin{array}{cl} 1 & \text{if } r=1 \\ 2 & \text{if } r \geq 2 \end{array} \right..$$
\end{claim}

\noindent
Since $\overline{K_r}$ does not have any edge, a median graph $M$ with $\overline{K_r}$ as its crossing graph does not have any two crossing $\Theta$-classes, which amounts to saying that $M$ is $\square$-free or equivalently that $M$ is a tree. Thus, the median degree of $\overline{K_r}$ coincides with the smallest degree of a tree with $r$ edges. This minimum is clearly acheived by a path $P_r$ of length $r$, whose degree is $1$ if $r=1$ and $2$ if $r \geq 2$. 
\end{proof}

\noindent
We conclude this section by mentioning the following estimate:

\begin{prop}\label{prop:AmalgamClique}
Let $A \ast_K B$ be a graph obtained by amalgamating two graphs $A$ and $B$ along a finite clique $K$. Then,
$$\mathrm{mdeg}(A \ast_K B) \leq \mathrm{mdeg}(A) + \mathrm{mdeg}(B) - |V(K)|.$$
\end{prop}

\begin{proof}
Let $M$ (resp.\ $N$) be a median graph of degree $\mathrm{mdeg}(A)$ (resp.\ $\mathrm{mdeg}(B)$) whose crossing graph coincides with $A$ (resp.\ $B$). The copy of $K$ in $A$ yields $|V(K)|$ $\Theta$-classes that pairwise cross. According to \cite[Theorem~4.14]{MR1347406} (see also \cite[Theorem~11]{MR1626579}), $M$ contains a cube $Q_A$ in which all these $\Theta$-classes cross. Similarly, $N$ contains a cube $Q_B$ in which the $|V(K)|$ $\Theta$-classes given by the copy of $K$ in $B$ pairwise cross. Because cubes in median graphs are convex, amalgamating $M$ and $N$ by identifying $Q_A$ and $Q_B$ yields a median graph $G$. Clearly, the crossing graph of $G$ coincides with $A \ast_K B$, hence
$$\mathrm{mdeg}(A \ast_K B) \leq \mathrm{deg}(G) = \mathrm{deg}(M) + \mathrm{deg}(N) - |V(K)| = \mathrm{mdeg}(A) + \mathrm{mdeg}(B) - |V(K)|,$$
which concludes the proof of our proposition. 
\end{proof}

\noindent
It is worth mentioning that, according to the next example, the inequality given by Proposition~\ref{prop:AmalgamClique} may not be an equality.

\begin{ex}
Consider the graph $C_5 \ast_{K_1} K_2$ obtained by gluing a cycle $C_5$ of length $5$ with a single edge $K_2$ along a single vertex $K_1$. Figure~\ref{AmalClique} yields a median graph $M$ of degree $5$ whose crossing graph is $C_5 \ast_{K_1} K_2$. Hence
$$\mathrm{mdeg}(C_5 \ast_{K_1} K_2) \leq 5 < 5+ 2 - 1 = \mathrm{mdeg}(C_5)+ \mathrm{mdeg}(K_2) - |V(K_1)|.$$
This shows that the inequality given by Proposition~\ref{prop:AmalgamClique} can be strict.
\end{ex}
\begin{figure}
\begin{center}
\includegraphics[width=0.8\linewidth]{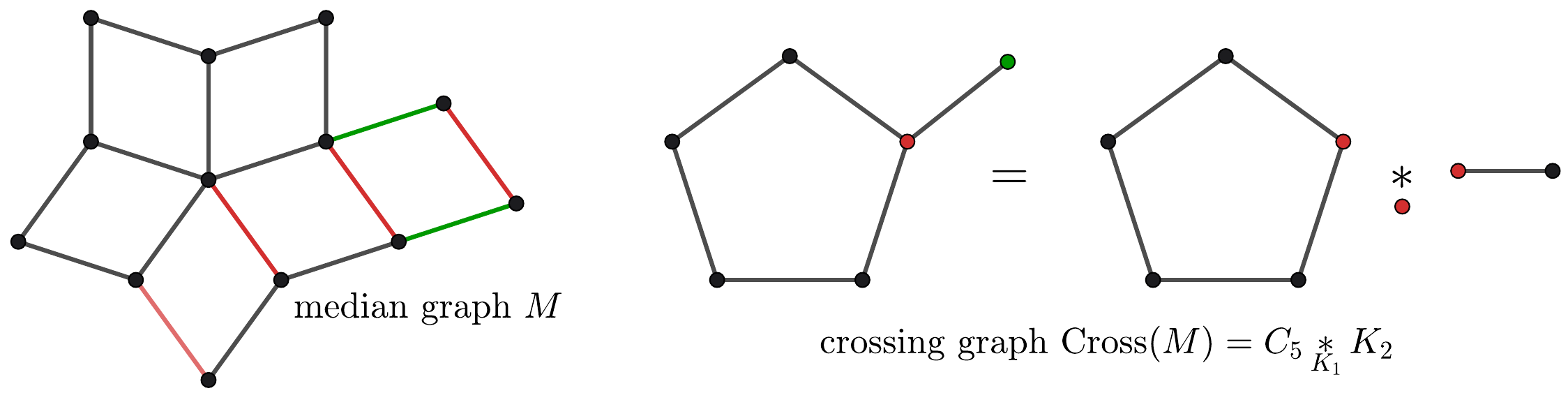}
\caption{An example showing that the inequality given by Proposition~\ref{prop:AmalgamClique} can be strict. }
\label{AmalClique}
\end{center}
\end{figure}

\noindent
As an application of Proposition~\ref{prop:AmalgamClique}, we record the following observation:

\begin{cor}
For every $n \geq 2$, if $K_n^-$ denotes the graph obtained from the complete graph $K_n$ on $n$ vertices by removing a single edge, then $\mathrm{mdeg}(K_n^-)= n.$
\end{cor}

\begin{proof}
If $n=2$, then $K_n^-$ is a pair of isolated vertices, so it is clear that $\mathrm{mdeg}(K_n^-)=2$ (the only median graph having $K_2^-$ as its crossing graph being a path of length $2$). If $n \geq 3$, then $K_n^-$ can be described as two $(n-1)$-cliques amalgamated along an $(n-2)$-clique. Proposition~\ref{prop:AmalgamClique}, combined with Corollary~\ref{cor:CliqueGirth}, then shows that
$$\mathrm{mdeg}(K_n^-) \leq \mathrm{mdeg}(K_{n-1}) + \mathrm{mdeg}(K_{n-1}) - |V(K_{n-2})| = 2(n-1) - (n-2) = n.$$
On the other hand, by noticing that $K_n^-$ is the crossing graph of two $(n-1)$-cubes amalgamated along an $(n-2)$-cube, we know that $\mathrm{mdeg}(K_n^-) \geq n$. The desired equality follows. 
\end{proof}

\section{Questions}\label{section:Questions}

\noindent
In this article, we have introduced, motivated, and initiated the study of median degrees of graphs. But many natural yet challenging problems remain to be investigated. A question we find particularly intriguing is:

\begin{question}\label{question:trees}
What is the median degree of a tree? 
\end{question}

\noindent
It seems that median degrees of trees can be arbitrarily large, but much smaller than the number of vertices. For instance, we saw in Proposition~\ref{prop:Grid} that the median degree of a path is always $\leq 3$. In fact, the median degree of a tree is always bounded above by its diameter plus $3$. Here is a sketch of proof; see Figure~\ref{Tree} for an example of the construction which we describe now. Let $T$ be a finite tree. Fix a $2$-vertex-colouring of $T$ and, for each vertex $v \in V(T)$, a \emph{local ordering}, i.e.\ a total ordering on the edges having $v$ as an endpoint. Draw the auxiliary tree $T'$ defined as follows: the vertices of $T'$ are the midpoints of the edges of $T$; for every vertex $v \in V(T)$, draw an edge between the midpoints of two edges having $v$ as an endpoint whenevery they appear successively along the order given by $v$. Notice that the $2$-vertex-colouring of $T$ induces a $2$-edge-colouring of $T'$. Say that the edges of $T'$ are coloured as green and red. Draw $T'$ on the plane such that every red edge is horizontal and every green edge is vertical, and such that left-to-right and up-to-down orderings agree with the local orderings. Naturally, $T'$ is the dual graph of a planar square complex whose $1$-skeleton yields a median graph $M$ having $T$ as its crossing graph. Observe that the degree of $M$ is bounded above by $2$ plus the length of the longest alternating red-green path in $T'$, hence $\mathrm{mdeg}(T) \leq \mathrm{deg}(M) \leq \mathrm{diam}(T')+2 \leq \mathrm{diam}(T)+3$. 
\begin{figure}
\begin{center}
\includegraphics[width=0.8\linewidth]{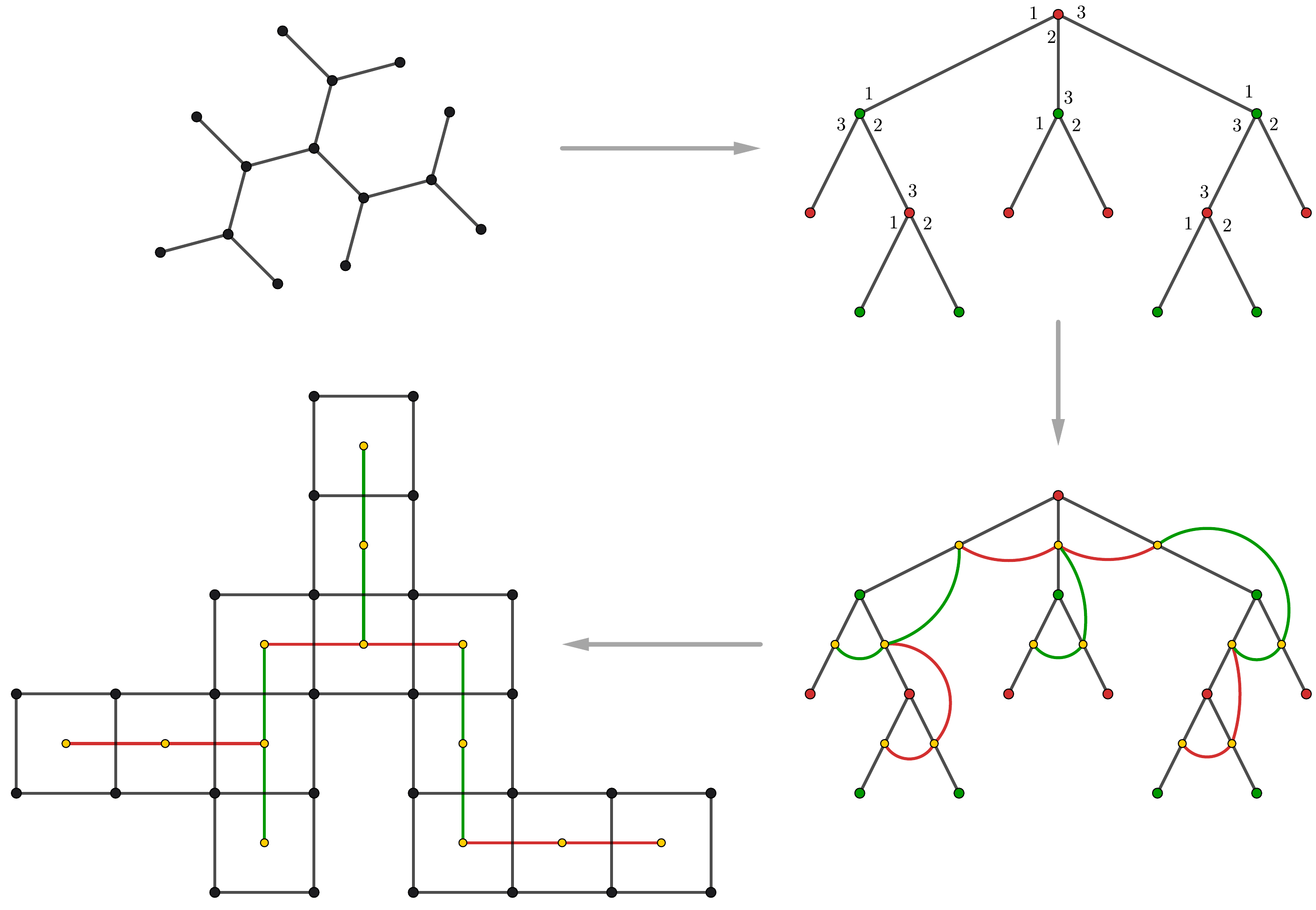}
\caption{Constructing a median graph with a given tree as its crossing graph.}
\label{Tree}
\end{center}
\end{figure}

\medskip \noindent
In the same direction as Question~\ref{question:trees}, we can ask:

\begin{question}
What is the median degree of a cactus graph? a block graph? a chordal graph?
\end{question}

\noindent
In another direction, a probably more delicate problem is:

\begin{question}
What is the median degree of a planar graph?
\end{question}

\noindent
As already mentioned in Section~\ref{prop:Grid}, it is natural to investigate how the median degree behaves under graph-theoretic operations. Propositions~\ref{prop:Join} and~\ref{prop:AmalgamClique} above deal with joins and amalgamations along cliques. We can also ask:

\begin{question}
What is the median degree of a $\square$-product? of a $\boxtimes$-product?
\end{question}

\noindent
Instead of looking for an explicit formula giving the median degree of a graph, which may not even exist for large families of graphs, it sounds reasonable to look for efficient algorithms that compute the median degrees of specific graphs, which leads to the following natural question:

\begin{question}
How efficient can we be in computing the median degree of a graph?
\end{question}

\noindent
In this article, we focused our interest, given a graph $X$, in the minimal degree of a graph in $\mathrm{Cross}^{-1}(X)$. But many other numerical invariants would be natural to investigate. For instance:

\begin{question}
Let $X$ be a graph. How many median graphs have crossing graph $X$?
\end{question}

\addcontentsline{toc}{section}{References}

\bibliographystyle{alpha}
{\footnotesize\bibliography{MedianDegree}}

\Address

%

\end{document}